\documentclass[12pt]{article}

\usepackage{fullpage}
\usepackage{cite}

\usepackage{ifthen}
\usepackage{colonequals}
\usepackage{amsmath}
\usepackage{amsthm}
\usepackage{array}
\usepackage{graphicx}
\usepackage{amssymb}
\usepackage{amsfonts}
\usepackage{mathrsfs}
\usepackage{tikz}
\usepackage{paralist}

\newcommand{\logimply}{\Rightarrow} 
\newcommand{\clog}{\mathrm{C}} 
\newcommand{\nat}{\mathbb{N}} 
\newcommand{\placeholder}{\cdot} 
\newcommand{\Hom}{\mathrm{Hom}} 
\newcommand{\lefthomvec}{\hom} 
\newcommand{\righthomvec}{\hom} 
\newcommand{\class}[1]{\mathscr{#1}} 
\newcommand{\classofgraphs}{\class{G}} 
\newcommand{\classoftrees}[1][]{\ifthenelse{\equal{#1}{}}{\class{T}}{\class{T}_{#1}}} 
\newcommand{\classofindeps}{\class{I}} 
\newcommand{\classofcycles}{\class{C}} 
\newcommand{\classofpaths}[1][]{\ifthenelse{\equal{#1}{}}{\class{P}}{\class{P}_{#1}}} 
\newcommand{\classofcliques}{\class{K}} 
\newcommand{\mathtext}[1]{\mbox{\rm#1}} 
\newcommand{\mathmode}[1]{\begin{math}#1\end{math}} 
\newcommand{\ds}[1]{{\displaystyle #1}}
\newcommand{\defas}{\colonequals} 
\newcommand{\etc}{\ldots} 
\newcommand{\seq}[2]{{(#1 \mid #2)}} 
\newcommand{\inv}[1]{#1^{-1}} 
\newcommand{\vertex}{\mathit{V}} 
\newcommand{\edge}{\mathit{E}} 
\newcommand{\dunion}{\oplus} 
\newcommand{\pathgraph}[1][]{\ifthenelse{\equal{#1}{}}{\mathit{P}}{\mathit{P}_{#1}}}
\newcommand{\cycle}[1][]{\ifthenelse{\equal{#1}{}}{\mathit{C}}{\mathit{C}_{#1}}}
\newcommand{\clique}[1][]{\ifthenelse{\equal{#1}{}}{\mathit{K}}{\mathit{K}_{#1}}}
\newcommand{\indep}[1][]{\ifthenelse{\equal{#1}{}}{\mathit{I}}{\mathit{I}_{#1}}} 
\newcommand{\chrom}{\chi} 
\newcommand{\sete}[1]{{\{#1\}}} 
\newcommand{\setm}[2]{{\{#1 \mid #2\}}} 
\newcommand{\sett}[2]{{\{#1 \mid \mathtext{#2}\}}} 
\newcommand{\intsc}{\mathop{\cap}} 
\newcommand{\union}{\mathop{\cup}} 
\newcommand{\bdunion}{\oplus} 
\newcommand{\card}[1]{{|#1|}} 
\newcommand{\folog}{\mathrm{FO}} 
\newcommand{\satis}{\models} 
\newcommand{\isom}{\cong} 
\newcommand{\Sur}{\mathrm{Sur}} 
\newcommand{\Inj}{\mathrm{Inj}} 
\newcommand{\Ext}{\mathrm{Ext}} 
\newcommand{\inj}{\mathrm{inj}} 
\newcommand{\sur}{\mathrm{sur}} 
\newcommand{\aut}{\mathrm{aut}} 
\newcommand{\Clogic}{\mathrm{C}} 

\newcommand{\ignore}[1]{}

\theoremstyle{plain}
\newtheorem{lemma}{Lemma}
\newtheorem{theorem}{Theorem}

\theoremstyle{definition}
\newtheorem{definition}{Definition}

\newtheorem{example}{Example}

\newtheorem{observation}{Observation}
\newtheorem{proposition}{Proposition}

\usetikzlibrary{automata,positioning}
\tikzset{
-,  
node distance=50pt, 
every state/.style={thick, fill=gray!10, scale=1}, 
initial text=$ $, 
}

\begin{document}
\title{On the Expressive Power of Homomorphism Counts}
\author{
Albert Atserias\thanks{Universitat Polit\`ecnica de Catalunya,
   Barcelona, Catalonia, Spain. Atserias' research partially
   supported by MICIN project PID2019-109137GB-C22 (PROOFS).}
\and
   Phokion G. Kolaitis\thanks{UC Santa Cruz and IBM Research, Santa
     Cruz, CA, USA. Kolaitis' research partially supported by NSF Grant
     IIS-1814152.}
\and
Wei-Lin Wu\thanks{UC Santa Cruz, Santa Cruz, CA, USA.}
}
\date{}
\maketitle

\begin{abstract}
  A classical result by Lov\'asz asserts that two graphs~$G$ and~$H$
  are isomorphic if and only if they have the same left profile, that is, for every graph~$F$, the number of homomorphisms
  from~$F$ to~$G$ coincides with the number of homomorphisms from~$F$
  to~$H$. Dvor{\'{a}}k and later on  Dell, Grohe, and Rattan showed that restrictions of the
  left profile to a class of graphs can capture several
  different relaxations of isomorphism, including equivalence in counting logics with a fixed number of variables (which contains fractional isomorphism as a special case) and co-spectrality (i.e., two graphs having the same characteristic polynomial).
   On the other side, a result by Chaudhuri
  and Vardi asserts that isomorphism is also captured by the
  right profile, that is, two graphs~$G$ and~$H$ are
  isomorphic if and only if for every graph~$F$, the number of
  homomorphisms from~$G$ to~$F$ coincides with the number of
  homomorphisms from~$H$ to~$F$. In this paper, we embark on a study
  of the restrictions of the right profile by
  investigating relaxations of isomorphism that can or cannot be
  captured by restricting the right profile to a fixed
  class of graphs. Our results unveil striking differences between the
  expressive power of the left profile and the
  right profile. We show that fractional isomorphism,
   equivalence  in counting logics with a fixed number of variables, and co-spectrality
  cannot be captured by restricting the right profile
  to a class of graphs.  In the opposite direction, we show
  that chromatic equivalence cannot be captured by restricting the
  left profile to a class of graphs, while, clearly, it
  can be captured by restricting the right profile to the
  class of all cliques.
\end{abstract}

\section{Introduction} \label{sec:intro}
Even though research on the graph isomorphism problem has spanned
several decades \cite{read1977graph} and in spite of significant recent
progress \cite{DBLP:conf/stoc/Babai16}, the exact complexity of the
graph isomorphism problem remains unknown. This state of affairs has
motivated the study of \emph{relaxations} of graph isomorphism, that
is, equivalence relations that are coarser than graph isomorphism.
Some well known such relaxations are based on indistinguishability of
two graphs via a heuristic for graph isomorphism, others are based on
indistinguishability of two graphs in some logical formalism, and
others are based on indistinguishability of two graphs via some graph
polynomial.

The~$k$-dimensional Weisfeiler-Leman method,~$k\geq 1$, is a
prominent example of a heuristic for graph isomorphism that
distinguishes some, but not all, non-isomorphic graphs. This method is
an iterative algorithm that assigns colors to~$k$-tuples of vertices
of a graph until a stable coloring is achieved
\cite{weisfeiler2006construction}. The~$1$-dimensional
Weisfeiler-Leman method is also known as the vertex refinement
algorithm and coincides with the fractional isomorphism test, which
asks for the existence of a rational solution to a 0-1 integer linear
program encoding the existence of an isomorphism between two
graphs. On the side of logic, the counting
logics~$\clog^k$,~$k \geq 1$, have found applications to the study of
graph isomorphism; each logic~$\clog^k$ augments the syntax of
first-order logic with counting quantifiers~$\exists i x$ asserting
that there are at least~$i$ distinct elements~$x$, but
the~$\clog^k$-formulas are required to have at most~$k$ distinct
variables. Cai, F\"urer, and Immerman
\cite{DBLP:journals/combinatorica/CaiFI92} have shown a tight
connection between the Weisfeiler-Leman method and the counting
logics, namely, for every~$k\geq 2$, two graphs are indistinguishable
by the~$(k-1)$-Weisfeiler-Leman method if and only if they satisfy
the same~$\clog^k$-sentences.

Graph polynomials are polynomials that
encapsulate one or more important invariants of the graph they are
associated with. Different graph polynomials give rise to different
relaxations of graph isomorphism. For example, the chromatic
polynomial~$\chrom(G,k)$ of a graph~$G$ returns the number of
all~$k$-colorings of~$G$ (see \cite{read1968introduction}); it also
gives rise to the \emph{chromatic equivalence} relation, which holds
between two graphs~$G$ and~$H$ precisely
when~$\chrom(G,k) = \chrom(H,k)$ holds for all $k \geq 1$.

Lov\'asz \cite{lovasz1967operations} showed that graph isomorphism can
be characterized in terms of left-homomorphism counts. If~$G$ and~$H$
are two graphs, then~$\hom(G,H)$ denotes the number of homomorphisms
from~$G$ to~$H$. Let~$\class{G}$ be the class of all graphs (graphs
are assumed to be finite, undirected and to have no loops and no multiples edges)
and let~$G$ be a graph. The \emph{left profile} of~$G$ is
the infinite
vector~$\lefthomvec(\class{G}, G) =\seq{\hom(F, G)}{F \in
  \class{G}}$. Lov\'asz's result asserts that two graphs~$G$ and~$H$
are isomorphic if and
only if~$\lefthomvec(\class{G}, G) = \lefthomvec(\class{G}, H)$.
Dvor{\'{a}}k \cite{DBLP:journals/jgt/Dvorak10} and subsequently
Dell,
Grohe, and Rattan \cite{DBLP:conf/icalp/DellGR18} (unaware of Dvor{\'{a}}k's earlier work) investigated
relaxations of graph isomorphism obtained by restricting the
left profile to a fixed class~$\class{F}$ of graphs, i.e.,
they considered vectors of the
form~$\lefthomvec(\class{F}, G) =\seq{\hom(F, G)}{F \in \class{F}}$
and the associated equivalence relation between graphs~$G$ and~$H$
defined by the
condition~$\lefthomvec(\class{F},G)= \lefthomvec(\class{F},H)$.
Through these investigations, it has been shown that the left profile restricted to the
class~$\class{T}$ of all trees captures fractional isomorphism, i.e.,
two graphs~$G$ and~$H$ are fractionally isomorphic if and only
if~$\lefthomvec(\class{T}, G) = \lefthomvec(\class{T}, H)$. It has also been shown
 that the left profile restricted to the
class~$\class{T}_{k-1}$ of all graphs of treewidth at most~$k$
captures indistinguishability in the counting logic~$\clog^k$ and,
hence, indistinguishability using the~$(k-1)$-dimensional
Weisfeiler-Leman method,~$k\geq 3$.  The study of restrictions of the
left profile was further pursued by B{\"{o}}ker et al.\
\cite{DBLP:conf/mfcs/BokerCGR19} and by Grohe et al.\
\cite{DBLP:conf/lics/Grohe20}.

On the other side and in the context of database theory, Chaudhuri and Vardi \cite{DBLP:conf/pods/ChaudhuriV93} showed that graph isomorphism can also be characterized in terms of right-homomorphism counts. The \emph{right profile} of a graph~$G$ is the infinite vector $\righthomvec(G,\class{G}) =\seq{\hom(G, F)}{F \in \class{G}}$. Chaudhuri and Vardi showed that two graphs~$G$ and~$H$ are isomorphic if and only if $\lefthomvec(G,\class{G}) = \righthomvec(H,\class{G})$. This result was rediscovered by Fisk \cite{DBLP:journals/dmgt/Fisk95}.

In this paper, we embark on a study of relaxations of graph
isomorphism arising by restricting the right profile to a
fixed class~$\class{F}$ of graphs
(vectors of the form~$\righthomvec(G, \class{F}) =\seq{\hom(G, F)}{F \in \class{F}}$),
i.e., we study equivalence
relations on two graphs~$G$ and~$H$ defined by the
condition~$\righthomvec(G,\class{F})= \righthomvec(H,\class{F})$.  Our
main aim is to compare the expressive power of restrictions of the
right profile vs. the expressive power of restrictions of
the left profile.  While Dvor{\'{a}}k \cite{DBLP:journals/jgt/Dvorak10} and Dell, Grohe, and Rattan
\cite{DBLP:conf/icalp/DellGR18} identified relaxations of graph
isomorphism that \emph{can} be captured by restrictions of the
left profile, here we identify relaxations of graph
isomorphism that \emph{cannot} be captured by \emph{any} restriction
of the right profile or by \emph{any} restriction of the
left profile.  It should be noted that Garijo, Goodall,
and Ne\v{s}et\v{r}il \cite{DBLP:journals/ejc/GarijoGN11} considered
restrictions of both the left profile and the
right profile, but used these restrictions to study a
different problem, namely, which graphs are uniquely determined by
certain graph polynomials, such as the aforementioned chromatic
polynomial.

We now present an overview of our technical results. First, we show
from first principles that there is no class~$\class{F}$ of graphs
such that the right profile restricted to~$\class{F}$
captures fractional isomorphism. After this, we use sophisticated
machinery to show that for every~$k\geq 3$, there is no
class~$\class{F}$ of graphs such that the right profile
restricted to~$\class{F}$ captures indistinguishability of graphs in
the counting logic~$\clog^k$ (which, as discussed earlier, is the same
as indistinguishability of graphs using the~$(k-1)$-dimensional
Weisfeiler-Leman method). As a matter of fact, we show a much
stronger inexpressibility result to the effect that if~$\equiv$ is an
equivalence relation on graphs that is finer
than~$\clog^1$-equivalence but coarser than~$\clog^k$-equivalence for
some~$k\geq 2$, then there is no class~$\class{F}$ of graphs such that
the right profile restricted to~$\class{F}$ captures the
equivalence relation~$\equiv$. The proof of this result uses a
combination of tools from constraint satisfaction and finite model
theory. The key technical tool is a definable version of
the~$H$-coloring dichotomy theorem of Hell and Ne\v{s}et\v{r}il
\cite{HellNesetril1990}, which asserts that, for every graph~$H$,
if~$H$ is not~$2$-colorable, then the~$H$-coloring problem is
NP-complete, while if~$H$ is~$2$-colorable, then the~$H$-coloring
problem is solvable in polynomial time.  The definable version of this dichotomy theorem
asserts that if~$H$ is not~$2$-colorable, then the~$H$-coloring
problem is not definable in the infinitary counting
logic~$\clog^{\omega}_{\infty\omega}$, while if~$H$ is~$2$-colorable,
then the~$H$-coloring problem is definable by the negation of a
Datalog sentence.

As a byproduct of our main inexpressibility result, we establish that
there is no class~$\class{F}$ of graphs such that the
right profile restricted to~$\class{F}$ captures
co-spectrality of graphs; by definition, two graphs are
\emph{co-spectral} if they have the same characteristic polynomial or,
equivalently, if their adjacency matrices have the same multiset of
eigenvalues. Note that, as pointed out by Dell, Grohe, and Rattan
\cite{DBLP:conf/icalp/DellGR18}, co-spectrality of graphs is captured
by the left profile restricted to the class~$\class{C}$ of
all cycles.

The results discussed so far concern limitations of the expressive
power of restrictions of the right profile. Switching to
the other side, we show that there is no class~$\class{F}$ of graphs
such that the left profile restricted to~$\class{F}$
captures chromatic equivalence. In contrast, chromatic equivalence is
captured by the right profile restricted to the
class~$\class{K}$ of all cliques. Note that the study of chromatically
equivalent graphs has a long history that starts with the work of
Birkhoff \cite{birkhoff1912determinant} in the early 20th
Century. Furthermore, chromatically equivalent graphs have several
invariants in common, including the same number of vertices, edges, triangles,
components, and girth (see \cite{DBLP:journals/tcs/Noy03}); it is worth pointing out that the first three of these invariants are left-homomorphism counts.

In this paper, we also obtain a common generalization of the
characterizations of graph isomorphism by Lov\'asz
\cite{lovasz1967operations} and by Chaudhuri and Vardi
\cite{DBLP:conf/pods/ChaudhuriV93}. Finally, at the conceptual level,
we discuss extensions of the framework studied here to graphs that may
have loops and weights on their vertices and edges, where the weights
are real numbers or, more generally, are elements of an arbitrary, but
fixed, semiring. This extension of the framework makes it possible to
capture several other relaxations of graph isomorphism arising from
fundamental graph polynomials, such as the cluster expansion polynomial
and the independence polynomial.

\section{Preliminaries and Basic Concepts} \label{sec:prelims}
Unless otherwise specified, all graphs are finite, undirected, and
simple, i.e., without self-loops or multi-edges. For a graph~$G$, we
write~$\vertex(G)$ for the set of \emph{vertices} (or \emph{nodes})
of~$G$ and~$\edge(G)$ for the set of \emph{edges} of~$G$. The vertices
are \emph{labelled}, i.e., they are elements of some fixed countable
set of labels, say~$\mathbb{N}$. If~$u$ and~$v$ are distinct vertices,
we write~$(u,v)$ to denote the edge with endpoints~$u$ and~$v$. Since
edges are undirected, the edges~$(u,v)$ and~$(v,u)$ are the same.

\begin{definition}
Let~$G$ and~$H$ be two graphs.
A function~$g: \vertex(G) \to \vertex(H)$ is an \emph{isomorphism} from~$G$ to~$H$
if~$g$ is a bijection, and for every two vertices~$u$ and $v$ in
$\vertex(G)$, we have that~$(u, v) \in \edge(G)$ if and only
if~$(g(u), g(v)) \in \edge(H)$.
We say that \emph{$G$ is isomorphic to~$H$} and we write~$G \isom H$ if there exists an isomorphism~$g$ from~$G$ to~$H$.
  A function~$h: \vertex(G) \to \vertex(H)$ is a \emph{homomorphism} from~$G$ to~$H$
  if for every two vertices~$u$ and~$v$
  in~$\vertex(G)$, we have that~$(u, v) \in \edge(G)$ implies~$(h(u),
  h(v)) \in \edge(H)$.
 We write~$h: G\rightarrow H$ to denote that~$h$ is a homomorphism
 from~$G$ to~$H$.
\end{definition}

In the sequel, whenever we say that~$\class{F}$ is a \emph{class} of
graphs, we mean that~$\class{F}$ is a non-empty collection of graphs
that is closed under isomorphisms, i.e., if~$G\in \class{F}$ and~$G$
is isomorphic to~$H$, then~$H \in \class{F}$.

We now introduce notation for various classes of graphs and for
particular graphs that will be used in the sequel. For~$n \geq 1$, we
write~$\cycle[n]$,~$\pathgraph[n]$,~$\clique[n]$ and~$\indep[n]$ to
denote the cycle with~$n$ vertices (and length~$n$), the path with~$n$
vertices (and length~$n-1$), the clique with~$n$ vertices, and the
independent set with~$n$ vertices, respectively. The
cycles~$\cycle[1]$ and~$\cycle[2]$ are called \emph{degenerate}.
Clearly,~$\indep[1] = \clique[1] = \pathgraph[1] = \cycle[1]$ is the
single-vertex graph and~$\clique[2] = \pathgraph[2] = \cycle[2]$ is
the single-edge graph (on two vertices).  We
write~$\classofgraphs$,~$\classoftrees$,~$\classofcycles$,~$\classofpaths$,~$\classofcliques$,
and~$\classofindeps$ for the classes of all graphs, all (unrooted)
trees, all cycles, all paths, all cliques, and all independent sets,
respectively.

\ignore{
We say that a graph $G$ is \emph{connected} if for every two distinct vertices $u, v \in \vertex(G)$, there is a path in $G$ from $u$ to $v$. For technical convenience, we regard the single-vertex graph $\indep[1]$ as connected.

The lemma below will be useful in later discussions.

\begin{lemma}
\label{conn_prsv_homo}
Let $G, H$ be two graphs and $h: G \to H$ be a homomorphism. If $G$ is connected, then so is the homomorphic image $h(G)$ of $G$ under $h$.
\end{lemma}
\begin{proof}
Let $G, H$ and $h$ be given as in the premise, and assume that $G$ is connected. Our goal is to show that the graph $h(G)$ is also connected. If $G = \indep[1]$ then $h(G)$ is isomorphic to $\indep[1]$ and hence is connected. Otherwise, $G$ contains at least an edge and so does $h(G)$, hence $h(G)$ contains at least two vertices. Given every two distinct vertices $u', v' \in \vertex(h(G))$, we choose arbitrary preimages $u, v \in \vertex(G)$ of $u', v'$ (so that $h(u) = u'$ and $h(v) = v'$), respectively. Then there is a path $p$ in $G$ from $u$ to $v$ since $G$ is connected. We take the homomorphic image of $p$ or more precisely, the sequence of vertices (in $\vertex(h(G))$) obtained by replacing in $p$ each occurrence of a vertex $w$ by its homomorphic image $h(w)$, then the resulting sequence is a (not necessarily simple) path $p'$ in $h(G)$ from $u'$ to $v'$.
\end{proof}

The \emph{disjoint union} of two graphs $G$ and $H$ is denoted $G \dunion H$. Note that for all graphs $F$, $G$ and $H$, we have $\hom(G \dunion H, F) = \hom(G, F) \times \hom(H, F)$.

}

The central notion in this paper is that of \emph{homomorphism count},
which we now introduce.

\begin{definition}
  Let~$G$ and~$H$ be two graphs.  We write~$\Hom(G, H)$ to denote the
  set of all homomorphisms from~$G$ to~$H$. In symbols,
$$\Hom(G, H) \defas \setm{h}{h: G \to H}.$$
We write~$\hom(G, H)$ to denote the number of homomorphisms from~$G$
to~$H$. In symbols,
$$\hom(G, H) \defas \card{\Hom(G, H)}.$$
\end{definition}

The homomorphism count can be viewed as a
function~$\hom: \classofgraphs \times \classofgraphs \to \nat$ from
pairs of graphs to natural numbers.  If we pick only one graph from
each isomorphism type, then the homomorphism count can be visualized
as the following infinite~$2$-dimensional matrix:
%
\begin{displaymath}
\begin{array}{c|cccc}
\hom(\placeholder, \placeholder) & G_1  & \cdots & G_j &\cdots \cr\hline
G_1 & \hom(G_1, G_1) & \cdots & \hom(G_1, G_j) & \cdots \cr
\vdots & \vdots & \vdots & \vdots & \ddots \cr
G_i & \hom(G_i, G_1) & \cdots & \hom(G_i, G_j) &  \cdots\cr
\vdots & \vdots & \vdots & \vdots &  \ddots\cr
\end{array}
\end{displaymath}
\ignore{
\begin{displaymath}
 \begin{array}{c|ccc}
 \hom(\placeholder, \placeholder) & G_1 & G_2 & \cdots  \cr\hline
 G_1 & \hom(G_1, G_1) & \hom(G_1, G_2) &  \cdots \cr
 G_2 & \hom(G_2, G_1) & \hom(G_2, G_2) & \cdots  \cr
 \vdots & \vdots & \vdots & \ddots \cr
 G_i & \hom(G_i, G_1) & \hom(G_i, G_2) & \cdots \cr
 \vdots & \vdots & \vdots &   \ddots\cr
 \end{array}
 \end{displaymath}
 }
\noindent For a graph~$G$, the column of this matrix indexed by~$G$ is referred
to as the \emph{left profile} of~$G$, and the row of this matrix
indexed by~$G$ is referred to as the \emph{right profile}
of~$G$. If~$\class{F}$ is a class of graphs, the
\emph{left profile of~$G$ restricted to~$\class{F}$} and
the \emph{right profile of~$G$ restricted to~$\class{F}$}
are defined as:
\begin{align*}
\lefthomvec(\class{F}, G) & \defas \seq{\hom(F, G)}{F \in \class{F}}, \\
\righthomvec(G, \class{F}) & \defas \seq{\hom(G, F)}{F \in \class{F}}.
\end{align*}
When~$\class{F}$ is the class~$\classofgraphs$ of all graphs, the
reference to the restriction to the class~$\class{F}$ is omitted from
the terminology.

We now state two well-known theorems by Lov\'asz
\cite{lovasz1967operations} and by Chaudhuri and Vardi
\cite{DBLP:conf/pods/ChaudhuriV93}, which characterize graph
isomorphism in terms of left profiles and
right profiles, respectively.

\begin{theorem}\label{lovasz_theorem} \cite{lovasz1967operations}
  For every two graphs~$G$ and~$H$, we have that~$G \isom H$ if and
  only
  if~$\lefthomvec(\classofgraphs, G) = \lefthomvec(\classofgraphs,
  H)$.
\end{theorem}

\begin{theorem}\label{chaudhuri_vardi_theorem} \cite{DBLP:conf/pods/ChaudhuriV93}
  For every two graphs~$G$ and~$H$, we have that~$G \isom H$ if and
  only
  if~$\righthomvec(G, \classofgraphs) = \righthomvec(H,
  \classofgraphs)$.
\end{theorem}

Since we have picked only one graph from each isomorphism type in
constructing the matrix, Theorem~\ref{lovasz_theorem} asserts that
\emph{no} two columns of the matrix are the same, while
Theorem~\ref{chaudhuri_vardi_theorem} asserts that \emph{no} two rows
of the matrix are the same. The `only if' direction of both these
theorems is trivial.

\section{Generalizing Lov\'asz and Chaudhuri-Vardi} \label{section_gen_lovasz_chaudhuri_vardi_theorems}
The preceding theorems by Lov\'asz and by Chaudhuri and Vardi
characterize graph isomorphism in terms of the left profile
and the right profile on the class~$\classofgraphs$
of all graphs. In this section, we generalize these result to
classes~$\class{F}$ of graphs that satisfy certain conditions.

\begin{definition} 
  Let~$G$ and~$H$ be two graphs.

  A homomorphism~$h: G \to H$ is
  \emph{injective} if the mapping~$h: \vertex(G) \to \vertex(H)$ is
  injective.  We write~$\inj(G, H)$ to denote the number of injective
  homomorphisms from~$G$ to~$H$.

  A homomorphism~$h: G \to H$ is
  \emph{surjective} if the image~$h(G)$ of~$G$ under~$h$ coincides
  with~$H$, i.e.,~$\vertex(h(G)) = \vertex(H)$
  and~$\edge(h(G)) = \edge(H)$. We write~$\sur(G, H)$ to denote the
  number of surjective homomorphisms from~$G$ onto~$H$.

   An
  isomorphism from~$G$ to~$G$ is called an \emph{automorphism} of~$G$.
  We write~$\aut(G)$ for the number of automorphisms of~$G$.
\end{definition}

\ignore{
Note that for every two graphs $G$ and $H$, $G$ is isomorphic to a subgraph of $H$ if and only if there is an injective homomorphism from $G$ into $H$.}

\begin{definition} Let~$\class{F}$ be a class of graphs.  We
  write~$\Inj(\class{F})$ to denote the class of all graphs~$G$ such
  that there is an injective homomorphism~$h: G \to F$ to some
  graph~$F\in \class{F}$.  We write~$\Sur(\class{F})$ to denote the
  class of all graphs~$G$ such that there is a surjective
  homomorphism~$h: F \to G$ from some graph~$F\in \class{F}$ onto~$G$.
  The \emph{extension class of~$\class{F}$}, denoted
  by~$\Ext(\class{F})$, is the intersection of~$\Inj(\class{F})$
  and~$\Sur(\class{F})$. In symbols,
  \begin{displaymath}
\Ext(\class{F}) \defas \Sur(\class{F}) \intsc \Inj(\class{F}).
\end{displaymath}
\end{definition}

It is obvious that $\class{F} \subseteq \Inj(\class{F})$ and
$\class{F} \subseteq \Sur(\class{F})$. Therefore, for every class
$\class{F}$ of graphs, we have that
$\class{F} \subseteq \Ext(\class{F})$.

The next result is a  simultaneous generalization of Theorem~\ref{lovasz_theorem}
and Theorem~\ref{chaudhuri_vardi_theorem}.

\begin{theorem}\label{theorem_general_lovasz_chaudhuri_vardi}
  Let~$\class{F}$ be a non-empty class of graphs. For every two
  graphs~$G$ and~$H$ in~$\class{F}$, the following statements are
  equivalent:
\begin{enumerate}[(1)] \itemsep=0pt
\item $G$ and $H$ are isomorphic.
\item $\lefthomvec(\Ext(\class{F}), G) = \lefthomvec(\Ext(\class{F}), H)$.
\item $\righthomvec(G, \Ext(\class{F})) = \righthomvec(H, \Ext(\class{F}))$.
\end{enumerate}
In particular, if~$\Ext(\class{F}) = \class{F}$, then for every two
graphs~$G$ and~$H$ in~$\class{F}$, we have that $G$ and $H$ are isomorphic if and only if their
left profiles restricted to~$\class{F}$ are equal, and
if and only if their right profiles restricted
to~$\class{F}$ are equal.
\end{theorem}
\begin{proof}
In the sequel, we assume that the class~$\classofgraphs$ of graphs is
linearly ordered first in increasing order of~$\card{\vertex(G)}$, then in
increasing order of~$\card{\edge(G)}$, and, finally, arbitrarily in case of a
tie; e.g., in lexicographical order of the rows of the adjacency
matrix. We write~$G < H$ to denote that~$G$ precedes~$H$ in this
ordering.
Moreover, we assume that in summations over a
class~$\class{F}$ of graphs, the summand~$F \in \class{F}$ runs over
the isomorphism types so that isomorphic graphs contribute only
once to the summation.

  The directions (1)~$\logimply$ (2) and (1)~$\logimply$ (3) are
  trivial. To prove (2)~$\logimply$ (1), let~$\class{F}$ be a
  non-empty class, and let~$G$ and~$H$ be two graphs in~$\class{F}$
  such
  that~$\lefthomvec(\Ext(\class{F}), G) = \lefthomvec(\Ext(\class{F}),
  H)$. Our goal is to show that~$G \isom H$, and we do so by arguing
  that~$\inj(G,H) > 0$ and~$\inj(H,G) > 0$.

  By induction on the position of~$D \in \Ext(\class{F})$ in the
  linear order~$<$ on~$\classofgraphs$, we show
  that~$\inj(D,G) = \inj(D,H)$. From this the goal will follow by
  setting~$D = H$, i.e.,~$\inj(H,G) = \inj(H,H) > 0$, and~$D = G$,
  i.e., $\inj(G,H) = \inj(G,G) > 0$. We start by noting that
  \begin{equation}
    \hom(D, G) = \sum_{E \in \classofgraphs} \sur(D, E) \cdot \inj(E, G) / \aut(E).
  \label{eqn:split}
  \end{equation}
  Next observe that, for~$E \in \classofgraphs$, if~$\sur(D, E) > 0$
  and~$\inj(E, G) > 0$, then~$E \in \Sur(\class{F})$
  since~$D \in \Ext(\class{F}) \subseteq \Sur(\class{F})$, and
  also~$E \in \Inj(\class{F})$ since~$G \in \class{F}$; that
  is,~$E \in \Ext(\class{F})$. This means that the sum
  in~\eqref{eqn:split} can be restricted to~$E \in
  \Ext(\class{F})$. Moreover, if~$E$ has more vertices or edges
  than~$D$, or if~$E$ and~$D$ have the same number of vertices and
  edges but are not isomorphic, then~$\sur(D,E) = 0$. Thus, the sum
  in~\eqref{eqn:split} can be restricted further to~$E < D$
  or~$E \cong D$. Since, by convention, sums over classes of graphs
  are restricted to isomorphism types, we get
  \begin{equation}
    \inj(D, G) + \sum_{\genfrac{}{}{0pt}{2}{E \in \Ext(\class{F}):}{E < D}} \sur(D, E) \cdot \inj(E, G) / \aut(E). \label{eqn:splitsplit}
    \end{equation}
    The same can be argued for~$H$ in place of~$G$. By induction
    hypothesis we have~$\inj(E,G) = \inj(E,H)$ for every~$E$ in the
    sum in~\eqref{eqn:splitsplit}, and by assumption we
    have~$\hom(D,G) = \hom(D,H)$. The conclusion is
    that~$\inj(D,G) = \inj(D,H)$, as was to be proved.

    The direction (3)~$\logimply$ (1) can be argued analogously, so we
    highlight only the key differences. This time we want to argue
    that~$\sur(G,H) > 0$ and~$\sur(H,G) > 0$, and we do so by arguing
    that~$\sur(G,D) = \sur(H,D)$ for every~$D \in
    \Ext(\class{F})$. The analogue of Equation~\eqref{eqn:split} is
  \begin{equation}
    \hom(G,D) = \sum_{E \in \classofgraphs} \sur(G, E) \cdot \inj(E, D) / \aut(E).
  \label{eqn:split2}
\end{equation}
Next, for~$E \in \classofgraphs$, if~$\sur(G,E) > 0$ and~$\inj(E,D) >
0$, then~$E \in \Sur(\class{F})$ since~$G \in \class{F}$, and also~$E
\in \Inj(\class{F})$ since~$D \in \Ext(\class{F}) \subseteq
\Inj(\class{F})$; that is,~$E \in \Ext(\class{F})$. Thus, the sum
in~\eqref{eqn:split2} can be restricted to~$\Ext(\class{F})$, and by
the same type of argument as before, further down to
  \begin{equation}
    \sur(G, D) + \sum_{\genfrac{}{}{0pt}{2}{E \in \Ext(\class{F}):}{E < D}} \sur(G, E) \cdot \inj(E, D) / \aut(E). \label{eqn:splitsplit2}
  \end{equation}
    The same can be argued for~$H$ in place of~$G$. By induction
    hypothesis we have~$\sur(G,E) = \sur(H,E)$ for every~$E$ in the
    sum in~\eqref{eqn:splitsplit2}, and by assumption we
    have~$\hom(G,D) = \hom(H,D)$. The conclusion is
    that~$\sur(G,D) = \sur(H,D)$, as was to be proved.
\end{proof}

Several remarks about Theorem \ref{theorem_general_lovasz_chaudhuri_vardi} are now in order.

First, Theorem
\ref{theorem_general_lovasz_chaudhuri_vardi} is indeed a common
generalization of Theorem \ref{lovasz_theorem} and Theorem
\ref{chaudhuri_vardi_theorem} because it is clear that the
condition~$\Ext(\classofgraphs) = \classofgraphs$ holds for the
class~$\classofgraphs$ of all graphs.

Second,
since~$\Ext(\class{F}) = \Sur(\class{F}) \intsc \Inj(\class{F})$, we
have that if~$\Sur(\class{F}) = \class{F}$  holds or
if~$\Inj(\class{F}) = \class{F}$ holds,
then~$\Ext(\class{F}) = \class{F}$ holds. Therefore,
 the special case $\Ext(\class{F}) = \class{F}$ mentioned in
Theorem~\ref{theorem_general_lovasz_chaudhuri_vardi} applies to every
class~$\class{F}$ of graphs such that~$\Sur(\class{F}) = \class{F}$
or~$\Inj(\class{F}) = \class{F}$.

Third, the special case $\Ext(\class{F}) = \class{F}$ mentioned in Theorem~\ref{theorem_general_lovasz_chaudhuri_vardi} holds for each
of the following classes~$\class{F}$ of graphs: the
class~$\classofpaths$ of all paths, the class~$\classoftrees$ of all
trees, the class~$\classoftrees[k]$ of graphs of treewidth at
most~$k$, the class~$\classofcliques$ of all cliques, the
class~$\class{X}_k$ \ignore{($\class{X}$ for `color') $\class{F}_k$} of all~$k$-colorable graphs, and the
class~$\class{D}_k$ of all graphs of degree at most~$k$. For the
classes~$\classoftrees[k]$,~$\class{X}_k$, \ignore{$\class{F}_k$} and~$\class{D}_k$, we
actually have~$\Inj(\class{F}) = \class{F}$, while for~$\classofcliques$,
we have~$\Sur(\classofcliques) = \classofcliques$.

To see that~$\Ext(\classoftrees) = \classoftrees$, it suffices to
show~$\Ext(\classoftrees) \subseteq \classoftrees$ since it is clear
that~$\classoftrees \subseteq \Ext(\classoftrees)$. First, observe
that trees are connected graphs and hence so are their homomorphic
images since the homomorphic image of a connected graph is
connected. In other words, the graphs in~$\Sur(\classoftrees)$ are
connected. Next,~$\Inj(\classoftrees)$ contains all subgraphs of
trees. Therefore, every graph
in~$\Ext(\classoftrees) = \Sur(\classoftrees) \intsc
\Inj(\classoftrees)$ is a connected subgraph of a tree, which must be
a tree. It follows that~$\Ext(\classoftrees) \subseteq \classoftrees$.
The same type of argument works for~$\classofpaths$.

The final remark of this section concerns the class~$\classoftrees[k]$
of graphs of treewidth at most~$k$. It follows from the above discussion that the
left profile~$\lefthomvec(\classoftrees[k], \placeholder)$
characterizes isomorphism on~$\classoftrees[k]$, namely, for
graphs~$G$ and~$H$ in~$\classoftrees[k]$ it holds that~$G$ and~$H$ are
isomorphic if and only if~$\lefthomvec(\classoftrees[k], G) =
\lefthomvec(\classoftrees[k], H)$. By Theorem 7 in \cite{ DBLP:journals/jgt/Dvorak10} and also by Theorems~1 and~3 in
\cite{DBLP:conf/icalp/DellGR18}, the
vector~$\lefthomvec(\classoftrees[k], \placeholder)$ also
characterizes~$\clog^{k + 1}$-equivalence; see also
Theorems~\ref{thm:tinhofer},~\ref{thm:CFI},~\ref{thm:grohe-frac-isom}
and~\ref{thm:grohe-ck-equiv} in this paper. Consequently, the
equivalence in~$\clog^{k + 1}$ determines isomorphism
on~$\classoftrees[k]$. This reproves and strengthens a result of Grohe
and Mari\~no \cite{DBLP:conf/icdt/GroheM99}; see Theorem~4 in that
paper.

\section{Right Profiles} \label{sec:right}
Let~$k\geq 1$ be a fixed positive integer. The~$k$-\emph{dimensional
  Weisfeiler-Leman method} \cite{weisfeiler2006construction} is an
iterative algorithm for graph isomorphism that assigns colors to
each~$k$-tuple of vertices. At each iteration, the algorithm refines
the color classes of the~$k$-tuples; the algorithm stops when a
\emph{stable} coloring is achieved, i.e., when no further refinement
of the color classes of the~$k$-tuples is possible. This method is a
heuristic for graph isomorphism in the sense that if
the~$k$-dimensional Weisfeiler-Leman method produces different stable
colorings when applied to two graphs~$G$ and~$H$, then the graphs~$G$
and~$H$ are not isomorphic. However, this is not a complete
isomorphism test because, for each~$k\geq 1$, there are non-isomorphic
graphs~$G_k$ and~$H_k$ on which the~$k$-dimensional Weisfeiler-Leman
method produces the same coloring. We
write~$G\equiv^k_{\mathrm{WL}} H$ to denote that the graphs~$G$
and~$H$ are indistinguishable via the~$k$-dimensional
Weisfeiler-Leman method. In view of the preceding
discussion,~$\equiv^k_{\mathrm{WL}}$ is an equivalence relation that
is strictly coarser than isomorphism.

We describe in more detail the~$1$-dimensional Weisfeiler-Leman
method.  Initially, all vertices have the same color. At each
iteration, two vertices~$u$ and~$v$ that were in the same color class
are assigned different colors if there is a color~$c$ such that the
number of neighbors of~$u$ that have color~$c$ is different from the
number of neighbors of~$v$ that have color~$c$. This process continues
until all vertices in the same color class have the same number of
neighbors in every color class.  Tinhofer
\cite{DBLP:journals/computing/Tinhofer86,DBLP:journals/dam/Tinhofer91}
has shown that the~$1$-dimensional Weisfeiler-Leman method amounts to
the \emph{fractional isomorphism} test. Let~$G=(V,E)$ and~$H=(V,E')$
be two graphs on the same set of vertices and let~$A$ and~$B$ be their
adjacency matrices.  Clearly,~$G$ and~$H$ are isomorphic if and only
if there is permutation matrix~$X$ such that~$AX = XB$, where a
permutation matrix is a square matrix of~$0$'s and~$1$'s such that
each row and each column contains exactly one~$1$. Consider now the
%
%
system of linear equations~$AX = XB$,~$X e =
e$,~$e^{\top} X = e^{\top}$, where~$e$ is the vector
of~$1$'s of length~$|V|$.  The graphs~$G$ and~$H$ are said to be
\emph{fractionally isomorphic} if this system has a non-negative
solution over the rational numbers, which is called a \emph{fractional
  isomorphism} between~$G$ and~$H$.

\begin{theorem} \label{thm:tinhofer}
  \cite{DBLP:journals/computing/Tinhofer86,DBLP:journals/dam/Tinhofer91}
  For every two graphs~$G$ and~$H$, we have that~$G$ and~$H$ are
  fractionally isomorphic if and only if~$G\equiv^1_{\rm WL} H$.
\end{theorem}

Cai, F\"urer, and Immerman \cite{DBLP:journals/combinatorica/CaiFI92}
characterized indistinguishability via the Weisfeiler-Leman method in
terms of indistinguishability in first-order logic with counting and a
fixed number of variables. A \emph{counting quantifier} is a
quantifier of the form~$\exists i x$, where~$i$ is a positive
integer. The meaning of a formula~$\exists i x \varphi(x)$ is that
there are at least~$i$ distinct elements~$x$ such that~$\varphi(x)$
holds. If~$\varphi$ is an FO-formula, then the
formula~$\exists i x \varphi(x)$ is clearly equivalent to an
FO-formula, so counting quantifiers do not add expressive power to
first-order logic. Counting quantifiers, however, become interesting
when we consider logics with a fixed number of distinct variables. For
every~$k\geq 1$, let~$\mathrm{FO}^k$ be the fragment of first-order
logic FO consisting of all FO-formulas with at most~$k$ distinct
variables, and let~$\Clogic^k$ be the logic obtained
from~$\mathrm{FO}^k$ by augmenting the syntax with all counting
quantifiers~$\exists i x$, where~$i\geq 1$. For example, the
formula~$$\exists i x (x=x) \land \neg \exists (i+1) x (x=x) \land
\forall x\forall y ( x \not= y \rightarrow E(x,y))$$ is
a~$\Clogic^2$-formula that is satisfied by a graph~$G$ precisely
when~$G$ is the clique~$K_i$ with~$i$ vertices.  We say that two
graphs~$G$ and~$H$ are~$\Clogic^k$-\emph{equivalent}, denoted
by~$G \equiv_\Clogic^k H$, if~$G$ and~$H$ satisfy the
same~$\Clogic^k$-sentences. Detailed information about counting
logics with finitely many variables can be found in the monograph
\cite{DBLP:books/cu/O2017}.

\begin{theorem} \label{thm:CFI}
  \cite{DBLP:journals/combinatorica/CaiFI92} For every~$k\geq 2$ and
  every two graphs~$G$ and~$H$, we have that~$G\equiv_{\Clogic}^k H$
  if and only if~$G \equiv^{k-1}_{\mathrm{WL}} H$.
\end{theorem}

An immediate consequence of Theorems \ref{thm:tinhofer} and
\ref{thm:CFI} is that two graphs are fractionally isomorphic if and
only if they are~$\Clogic^2$-equivalent. The following family of
pairwise fractionally isomorphic graphs will be of interest to us in
the sequel.

\begin{example}\label{example_frac_isom_graphs}
  For every~$n \geq 3$, consider the graphs~$G_n$ and~$H_n$ defined as
  follows:~$G_n \defas A_n \dunion B_n$, where $\dunion$ denotes the
  disjoint union of two graphs, $A_n$ and~$B_n$ are two
  isomorphic copies of~$\clique[n]$, and~$H_n$ is formed from~$G_n$ by
  first removing an arbitrary edge~$(a_1, a_2) \in \edge(A_n)$ and an
  arbitrary edge~$(b_1, b_2) \in \edge(B_n)$, and then adding two
  edges~$(a_1, b_1), (a_2, b_2)$. The graphs~$G_3, H_3$ are depicted in
  Figure~\ref{fig_frac_isom}.
\begin{figure}[!t]
\centering
\begin{tikzpicture}[every node/.style={draw=black,thick,circle,inner sep=0pt}]
\node[circle,inner sep=2pt,minimum size=5pt] (a3) {$a_3$};
\path (a3) ++(60:35pt) node (a1) [circle,inner sep=2pt,minimum size=5pt] {$a_1$};
\path (a1) ++(0:35pt) node (b1) [circle,inner sep=2pt,minimum size=5pt] {$b_1$};
\path (b1) ++(-60:35pt) node (b3) [circle,inner sep=2pt,minimum size=5pt] {$b_3$};
\path (b3) ++(-120:35pt) node (b2) [circle,inner sep=2pt,minimum size=5pt] {$b_2$};
\path (a3) ++(-60:35pt) node (a2) [circle,inner sep=2pt,minimum size=5pt] {$a_2$};
\path (b3) ++(0:70pt) node (c3) [circle,inner sep=2pt,minimum size=5pt] {$a_3$};
\path (c3) ++(60:35pt) node (c1) [circle,inner sep=2pt,minimum size=5pt] {$a_1$};
\path (c1) ++(0:35pt) node (d1) [circle,inner sep=2pt,minimum size=5pt] {$b_1$};
\path (d1) ++(-60:35pt) node (d3) [circle,inner sep=2pt,minimum size=5pt] {$b_3$};
\path (d3) ++(-120:35pt) node (d2) [circle,inner sep=2pt,minimum size=5pt] {$b_2$};
\path (d2) ++(180:35pt) node (c2) [circle,inner sep=2pt,minimum size=5pt] {$a_2$};
\draw
(a3) edge (a1)
(a1) edge (a2)
(a2) edge (a3)
(b1) edge (b3)
(b3) edge (b2)
(b2) edge (b1)
(c3) edge (c1)
(c1) edge (d1)
(d1) edge (d3)
(d3) edge (d2)
(d2) edge (c2)
(c2) edge (c3)
;
\end{tikzpicture}
\caption{Two fractionally isomorphic graphs, $G_3$ (left) and $H_3$ (right).}
\label{fig_frac_isom}
\end{figure}
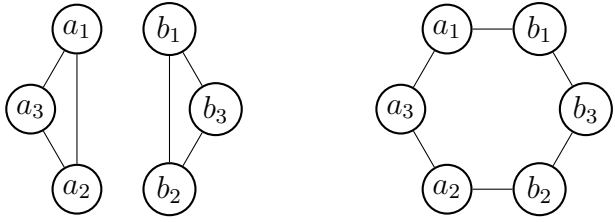

The graphs~$G_n$ and~$H_n$ are fractionally isomorphic. The reason is
that both are regular (each vertex has~$n-1$ neighbors) and they have
the same number of vertices and edges, thus the~$1$-dimensional
Weisfeiler-Leman method terminates at the first step with all
vertices assigned the same color.
\end{example}

\subsection{Fractional Isomorphism vs.\ Right Profiles} \label{subsec:fractional}

Dell, Grohe, and Rattan \cite{DBLP:conf/icalp/DellGR18} showed that
 fractional isomorphism  can be characterized by restricting the left profile.

 \begin{theorem} \label{thm:grohe-frac-isom}
   \cite{DBLP:conf/icalp/DellGR18} For every two graphs~$G$ and~$H$,
   we have that~$G$ and~$H$ are fractionally isomorphic if and only
   if~$\lefthomvec(\class{T}, G) = \lefthomvec(\class{T}, H)$,
   where~$\class{T}$ is the class of all trees.
 \end{theorem}

 In contrast, we show that fractional isomorphism cannot be
 characterized by restricting the right profile. While
 this result will follow from the more general
 Theorem~\ref{thm:general} proved in the next subsection, the proof for
 the special case of fractional isomorphism is more elementary and
 self-contained, so we include it here.

 \begin{theorem} \label{thm:frac-isom-right} There is no
   class~$\class{F}$ of graphs such that for every two graphs~$G$
   and~$H$, we have that~$G$ and~$H$ are fractionally isomorphic if and
   only if~$\righthomvec(G, \class{F}) = \righthomvec(H, \class{F})$.
\end{theorem}

\begin{proof} Towards a contradiction, assume that such a class
$\class{F}$ exists. We distinguish the following two cases.

\emph{Case 1.}~$\class{F} \subseteq \classofindeps$, i.e., each graph
in~$\class{F}$ is an independent set. This leads immediately to a
contradiction. Indeed, the cliques~$K_2$ and~$K_3$ are not
fractionally isomorphic, but~$\hom(K_2,\class{F})=\hom(K_3,\class{F})$
because, for every~$F \in \class{F}$, we have
that~$\hom(K_2,F)=0=\hom(K_3,F)$.

\emph{Case 2.}~$\class{F} \not \subseteq \classofindeps$, which means
that there is at least one graph~$F\in \class{F}$ such that~$F$
contains~$K_2$ as a subgraph. We now bring into the picture the family
of graphs~$G_n$ and~$H_n$, with~$n\geq 3$, considered in Example
\ref{example_frac_isom_graphs}. Since~$G_n$ and~$H_n$ are fractionally
isomorphic graphs, the hypothesis for the class~$\class{F}$ implies
that~$\hom(G_n,F) = \hom(H_n,F)$, for every~$n\geq 3$.  Furthermore,
we claim that the following properties hold for every graph~$D$ and
every~$n\geq 3$. We use the notation~$G \subseteq H$ to denote the
fact that~$H$ contains~$G$ as a subgraph.

\medskip
\emph{P1.}~ $\hom(G_n, D) > 0$ if and only if $\clique[n] \subseteq D$.

\emph{P2.}~ $\hom(H_n, D) > 0$ if and only if $\clique[n-1] \subseteq D$.
%
\medskip

For the first property, observe that if~$h: G_n \to D$ is a
homomorphism, then the homomorphic image~$h(G_n)$ must
contain~$\clique[n]$ as a subgraph because~$\clique[n]$ is a subgraph
of~$G_n$ and is preserved under~$h$; it follows that~$D$
contains~$\clique[n]$ as a subgraph, too. Conversely, if~$D$
contains~$\clique[n]$ as a subgraph, then there exists a
homomorphism~$h': G_n \to D$ that maps each of the two disjoint copies~$A_n$
and~$B_n$ of~$\clique[n]$ in~$G_n$ to the~$\clique[n]$ in~$D$.

For the second property, the `only if' direction can be argued in the
same way as that in the first part using the fact that~$\clique[n-1]$
is a subgraph of~$H_n$. For the `if' direction, assume that~$D$
contains~$\clique[n - 1]$ as a subgraph; then there is a
homomorphism~$h': H_n \to D$ that maps~$H_n$ onto~$\clique[n - 1]$
in~$D$ where~$h(a_1) = h(a_2) = h(b_3)$ and~$h(b_1) = h(b_2) = h(a_3)$
for some~$a_3 \in \vertex(A_n) \setminus \sete{a_1, a_2}$ and
some~$b_3 \in \vertex(B_n) \setminus \sete{b_1, b_2}$.

Using these two properties, we will show by induction on $n$ that~$F$
contains~$\clique[n]$ as a subgraph, for every~$n \geq 2$, which is
absurd since~$\card{\vertex(F)}$ is finite. The base case~$(n=2)$ is
just the assumption that~$F$ contains~$\clique[2]$ as a subgraph. For
the inductive case, let~$n \geq 2$ and suppose that~$F$
contains~$\clique[n]$ as a subgraph. By taking~$D=F$ and applying the
second property, we have that~$\hom(H_{n + 1}, F) > 0$,
hence~$\hom(G_{n + 1}, F) > 0$
since~$\hom(G_{n + 1}, F) = \hom(H_{n + 1}, F)$. By applying the first
property, we conclude that~$F$ contains~$\clique[n + 1]$ as a
subgraph. This completes the proof of the theorem.
\end{proof}

\subsection{$\Clogic^k$-Equivalence vs.\ Right Profiles} \label{subsec:ck-equiv}
In \cite{DBLP:journals/jgt/Dvorak10} and also in
\cite{DBLP:conf/icalp/DellGR18}   it was shown
that, for every~$k\geq 3$, equivalence in~$\Clogic^k$ can be
characterized by restricting the left profile.

\begin{theorem} \label{thm:grohe-ck-equiv}
  \cite{DBLP:journals/jgt/Dvorak10,DBLP:conf/icalp/DellGR18} For every integer~$k\geq 3$ and
  every two graphs~$G$ and~$H$, we have that~$G \equiv^k_{\Clogic} H$
  if and only
  if~$\lefthomvec(\class{T}_{k-1}, G) = \lefthomvec(\class{T}_{k-1},
  H)$, where~$\class{T}_{k-1}$ is the class of all graphs of treewidth
  at most~$k-1$.
 \end{theorem}

 In contrast, we show that~$\Clogic^k$-equivalence,~$k\geq 2$, cannot
 be characterized by restricting the right profile. Since
 fractional isomorphism and~$\Clogic^2$-equivalence are the same, the
 case~$k = 2$ coincides with the statement of
 Theorem~\ref{thm:frac-isom-right}. We chose to state Theorem~\ref{thm:frac-isom-right} earlier and present its proof, because that proof is self-contained, unlike the proof of Theorem \ref{thm:Ck-equivalence} below, which requires sophisticated tools.

\begin{theorem} \label{thm:Ck-equivalence} For every~$k \geq 2$,
    there is no class~$\class{F}$ of graphs such that for every two
    graphs~$G$ and~$H$, we have that~$G \equiv^k_\Clogic H$ if
    and only
    if~$\righthomvec(G, \class{F}) = \righthomvec(H, \class{F})$.
\end{theorem}

This theorem will be an immediate consequence of the following much more
general result, which states that  if an equivalence relation on graphs
 interpolates between~$\equiv^k_\Clogic$ and~$\equiv^1_\Clogic$
for some~$k \geq 2$, then it cannot be characterized by restricting the right profile.
We note that, for loopless graphs (as is our
case), two graphs are~$\equiv^1_\Clogic$-equivalent if and only if they
have the same number of vertices.

\begin{theorem} \label{thm:general} For every equivalence
  relation~$\,\equiv\,$ on graphs that is finer
  than~$\,\equiv^1_\Clogic\,$ and coarser than~$\,\equiv^k_\Clogic\,$
  for some~$k \geq 2$, there is no class~$\class{F}$ of graphs such
  that for every two graphs~$G$ and~$H$, we have that~$G \equiv H$ if
  and only
  if~$\righthomvec(G, \class{F}) = \righthomvec(H, \class{F})$.
 \end{theorem}

To prove this we need to bring in some tools from the theory of
constraint satisfaction. We focus on the special case of this
theory that applies to graphs since this is all we need.

Fix a graph~$H$. A graph~$G$ is called~$H$-\emph{colorable} if there is a
homomorphism from~$G$ to~$H$. The~$H$-coloring problem asks:  given a
graph  as input,  is it~$H$-colorable? The name of the problem
reflects the fact that the~$\clique[k]$-coloring problem is the same
as that of deciding whether a given graph as input has a proper
coloring with~$k$ colors. The celebrated $H$-Coloring Dichotomy Theorem,
due to Hell and Ne\v{s}et\v{r}il \cite{HellNesetril1990}, asserts that
the~$H$-coloring problem is NP-complete if~$H$ is non-$2$-colorable, and
solvable in polynomial time if $H$ is $2$-colorable. Here, we will use what we call the
\emph{Definable $H$-Coloring Dichotomy Theorem}. The logics that are
mentioned in its statement are defined immediately following it.

\begin{theorem} \label{thm:definableHcoloring} Let~$H$ be a graph. The
  class of~$H$-colorable graphs is not definable in the
  logic~$\Clogic^{\omega}_{\infty\omega}$ unless~$H$ is 2-colorable, in
  which case it is definable by the negation of a~$3$-Datalog
  sentence.
\end{theorem}

The logic~$\Clogic^{\omega}_{\infty\omega}$ denotes the union of the
logics~$\Clogic^k_{\infty\omega}$ as~$k$ ranges over the positive
integers, where $\Clogic^k_{\infty\omega}$ is the logic that is obtained
from~$\Clogic^k$ by augmenting the syntax with infinitary disjunctions
and conjunctions.  It is known that~$\Clogic^k_{\infty\omega}$, and
even its counting-free, existential, positive
fragment~$\exists\mathrm{L}^k_{\infty\omega}$, is at least as
expressive as the~$k$-variable fragment~$k$-Datalog of Datalog; see
Theorem~4.1 in~\cite{KolaitisVardi2000}.

We claim that Theorem~\ref{thm:definableHcoloring} is implicit in the
literature and follows by combining known results. Concretely,
Theorem~\ref{thm:definableHcoloring} can be seen to follow by
combining the implication~(c) to~(b) of Theorem~1
in~\cite{Bulatov2005} with Corollary~23
in~\cite{AtseriasBulatovDawar2009}. Unfortunately, showing that the
output of Theorem~1 in the first of these references can serve as
input for Corollary~23 in the second reference is not straightforward
because it would require us to introduce the notions of Tame
Congruence Theory that are used in the statement of Corollary~23
in~\cite{AtseriasBulatovDawar2009}. For this reason, and also because
we were not able to find a concrete reference
where~Theorem~\ref{thm:definableHcoloring} is stated and proved, we
provide some details here.

The part of Theorem~\ref{thm:definableHcoloring} that states that the
class of~$H$-colorable graphs is definable by the negation of
a~$3$-Datalog sentence when~$H$ is 2-colorable is well-known: If~$H$ is an independent set, then the class of~$H$-colorable graphs coincides with class~$\classofindeps$ of independent sets,
which is clearly definable as the negation of a
Datalog program even with two variables. If~$H$ has at least one edge
but is 2-colorable, then the class of~$H$-colorable graphs is precisely
the class of 2-colorable graphs, which is one of the canonical examples
of definability in (the negation of) Datalog by stating that the graph
contains an odd cycle. A Datalog program for this that uses four
variables was given in Section~4.1 of~\cite{KolaitisVardi2000}, but a
simple optimization shows that three variables are enough. It is also
easy to see that three variables are necessary.

Next, we focus on the non-2-colorable case of
Theorem~\ref{thm:definableHcoloring}. The development of the theory is
facilitated by the use of finite relational structures of richer
vocabularies. Graphs are seen as finite structures of a vocabulary
that has a single binary relation symbol whose interpretation is
symmetric and irreflexive.  Let~$A$ be a finite relational structure
with domain~$D(A)$.  If~$S \subseteq D(A)$ is a subset of the domain
of~$A$, then we write~$A_S$ for the substructure of~$A$ induced
by~$S$; i.e., the domain of~$A_S$ is~$S$, and its relations are the
sets of tuples of elements in~$S$ that are tuples in the corresponding
relation in~$A$.  If~$E \subseteq D(A) \times D(A)$ is an equivalence
relation on the domain of~$A$, then we write~$A/E$ for the quotient
structure of~$A$; i.e., the domain of~$A/E$ is the set
of~$E$-equivalence classes~$[a]$ of elements~$a$ in the domain of~$A$,
and its relations are the sets of tuples of equivalence
classes~$([a_1],\ldots,[a_r])$ such that~$(a_1,\ldots,a_r)$ is a tuple
in the corresponding relation in~$A$. We write~$A^c$ for the
singleton-expansion of~$A$; i.e., the expansion of~$A$ with a new
unary relation symbol~$P_a$ interpreted by the singleton set~$\{a\}$
for each element~$a$ in the domain of~$A$. A \emph{primitive positive} formula, or pp-formula, is a
first-order formula of the
form~$\exists y_1 \cdots \exists y_s
\psi(x_1,\ldots,x_r,y_1,\ldots,y_s)$, where~$\psi$ is a conjunction of
atomic formulas on the variables~$x_1,\ldots,x_r$
and~$y_1,\ldots,y_s$; among the atomic formulas we include equalities
between variables. Let~$R$ be a relation on~$D(A)$ of arity~$r$.  We say
that a~$A$ \emph{pp-defines}~$R$ if there exists a
pp-formula~$\varphi(x_1,\ldots,x_r)$ such
that~$R = \{(a_1,\ldots,a_r) \in D(A)^r : A \models
\varphi(x_1/a_1,\ldots,x_r/a_r) \}$. We also say that~$R$ is \emph{pp-definable}
in~$A$. The following is the algebraic version of the $H$-Coloring
Dichotomy Theorem due to Bulatov; see the implication~(c) to~(b) in
Theorem~1 in~\cite{Bulatov2005}.

\begin{theorem} \cite{Bulatov2005} \label{thm:hellnesetrilbulatov}
  If~$H$ is a non-2-colorable core graph, then there exists a subset~$S$
  of vertices of~$H$ and an equivalence
  relation~$E \subseteq S \times S$ on~$S$ such that the following
  three statements hold: (1)~$H^c$ pp-defines~$S$, (2)~$H_S^c$
  pp-defines~$E$, and (3)~$H_S/E$ is isomorphic to~$K_3$.
\end{theorem}

A graph $G$ is a \emph{core} if $G$ does not have a retraction to a
proper subgraph $H$, i.e., a homomorphism from $V(G)$ to $V(H)$ that
is the identity on $V(H)$ (see \cite{DBLP:books/daglib/0013017} for
detailed discussion of retractions and cores).  Also, to clarify, the
statement~(c) in Theorem~1 in~\cite{Bulatov2005} says that the
equivalence relation~$E$ is pp-definable in~$H^c$, but the proof
actually shows that it is pp-definable in~$H_S^c$. This is, in
principle, a stronger statement, and it is what we actually need
below. We note that the implications~(a) to~(b) and~(a) to~(c) in
Theorem~1 in~\cite{Bulatov2005} tacitly assume that P~$\not=$ NP, but
we do not use that part of the theorem.

The next ingredient in the proof is the concept of logical
reducibility between classes of finite structures. We refer the reader
to Definition~1 in~\cite{AtseriasBulatovDawar2009} for the definition
of Datalog-reducibility, denoted by~$\leq_{\mathrm{datalog}}$ there
and by~$\leq_{\mathrm{d}}$ here, for brevity. We need only one
property about~$\leq_{\mathrm{d}}$, namely, that if~$\class{B}$ is a
class of finite structures that
is~$\Clogic^\omega_{\infty\omega}$-definable and~$\class{A}$ is
another class of finite structures such
that~$\class{A} \leq_{\mathrm{d}} \class{B}$, then also~$\class{A}$
is~$\Clogic^\omega_{\infty\omega}$-definable. This follows from the
aforementioned fact that Datalog is a fragment
of~$\exists\mathrm{L}^\omega_{\infty\omega} \subseteq
\Clogic^\omega_{\infty\omega}$. For a finite structure~$A$,
let~$\mathrm{CSP}(A)$ denote the class of structures of the same
vocabulary as~$A$ that have a homomorphism to~$A$. The following was
proved in \cite{BulatovJeavonsKrokhin2005} for polynomial-time
reducibility, and in~\cite{AtseriasBulatovDawar2009}
for~$\leq_{\mathrm{d}}$ reducibility.

\begin{theorem}
  \cite{AtseriasBulatovDawar2009} \label{thm:reducibilities} Let~$A$
  be a finite relational structure with at least two elements,
  let~$S \subseteq D(A)$ be a subset of the domain of~$A$, and
  let~$E \subseteq D(A) \times D(A)$ be an equivalence relation on the
  domain of~$A$. The following hold:
  \begin{enumerate}[(1)] \itemsep=0pt
  \item $\mathrm{CSP}(A) \leq_{\mathrm{d}} \mathrm{CSP}(A^c)$.
  \item If $A$ pp-defines $E$, then
    $\mathrm{CSP}(A/E) \leq_{\mathrm{d}} \mathrm{CSP}(A)$.
  \item If~$A$ pp-defines~$S$,
    then~$\mathrm{CSP}(A_S) \leq_{\mathrm{d}} \mathrm{CSP}(A)$.
  \item If $A$ is a core, then
    $\mathrm{CSP}(A^c) \leq_{\mathrm{d}} \mathrm{CSP}(A)$.
  \end{enumerate}
\end{theorem}

\noindent In the statement of part~(4), the core of a relational
structure is the straightforward generalization of the core of a graph
to structures of arbitrary relational vocabularies. To see how
Theorem~\ref{thm:reducibilities} follows from the statements
in~\cite{AtseriasBulatovDawar2009}, note that~(1) follows from a
special case of Lemma~11 there, that~(2) and~(3) are special cases of
Theorem~18 there, and that~(4) follows from the second part of
Lemma~19 there.

Theorems~\ref{thm:reducibilities} and~\ref{thm:hellnesetrilbulatov} imply 
 that if~$H$ is non-2-colorable,
then~$\mathrm{CSP}(K_3) \leq_{\mathrm{d}} \mathrm{CSP}(H'_S/E)
\leq_{\mathrm{d}} \mathrm{CSP}((H'_S/E)^c) \leq_{\mathrm{d}}
\mathrm{CSP}(H'^c_S/E) \leq_{\mathrm{d}} \mathrm{CSP}(H'^c_S)
\leq_{\mathrm{d}} \mathrm{CSP}(H'^c) \leq_{\mathrm{d}}
\mathrm{CSP}(H') \leq_{\mathrm{d}} \mathrm{CSP}(H)$, where~$H'$ is the
core of~$H$. Step~3 is not from Theorem~\ref{thm:reducibilities}; it
holds because~$H'^c_S/E$ is~$(H'_S/E)^c$ with some additional
(singleton) relations. The last step holds because~$H$ and~$H'$ are
homomorphically equivalent. The final link that yields the proof of
Theorem~\ref{thm:definableHcoloring} is the following well-known
result of Dawar, which was obtained by adapting the main result
in~\cite{DBLP:journals/combinatorica/CaiFI92}; see Theorem~4.11 and
Remark~4.12 in~\cite{Dawar1998}.

\begin{theorem} \cite{Dawar1998} The class of~$3$-colorable graphs is
  not~$\Clogic^\omega_{\infty\omega}$-definable.
\end{theorem}

By the downwards preservation
of~$\Clogic^\omega_{\infty\omega}$-definability
through~$\leq_{\mathrm{d}}$, this completes the proof of
Theorem~\ref{thm:definableHcoloring}. We are now ready to prove
Theorem~\ref{thm:general} and, hence,
Theorem~\ref{thm:Ck-equivalence}.

\begin{proof}[Proof of Theorem~\ref{thm:general}]
  Fix an equivalence relation~$\equiv$ that is finer
  than~$\equiv^1_\Clogic$ and is coarser than~$\equiv^k_\Clogic$ for
  some~$k \geq 2$.  Towards a contradiction, assume that there exists
  a class of graphs~$\class{F}$ such that~$G \equiv H$ holds if and
  only if~$\hom(G,\class{F}) = \hom(H,\class{F})$ holds. We
  distinguish the following two cases:

  \emph{Case 1.}~All graphs in~$\class{F}$ are 2-colorable.  This leads
  immediately to a contradiction. Indeed, the cliques~$K_3$ and~$K_4$
  have different sizes, so they are not~$\Clogic^1$-equivalent, hence
  they are not~$\equiv$-equivalent. But~$\hom(K_3,F)=\hom(K_4,F)=0$ for every
  2-colorable graph~$F$, so~$\hom(K_3,\class{F})=\hom(K_4,\class{F})$.

  \emph{Case 2.}~Some graph in~$\class{F}$ is non-2-colorable.
  Let~$H \in \class{F}$ be such a graph. By
  Theorem~\ref{thm:definableHcoloring}, the class of~$H$-colorable
  graphs is not~$\Clogic^k_{\infty\omega}$-definable. By Corollary~2.4
  in~\cite{DBLP:books/cu/O2017} this means that the class
  of~$H$-colorable graphs is not preserved
  by~$\equiv^k_{\Clogic}$-equivalence, i.e., there exists a pair of
  graphs~$G_0$ and~$G_1$ such that~$G_0 \equiv^k_{\Clogic} G_1$
  with~$\hom(G_0,H) = 0$ and~$\hom(G_1,H) \not= 0$. In
  particular, $G_0 \equiv G_1$ and~$\hom(G_0,H) \not= \hom(G_1,H)$,
  against the fact that~$H$ is in~$\class{F}$ and the assumption
  on~$\class{F}$.
\end{proof}

\subsection{Cospectrality vs.\ Right Profiles}

The \emph{characteristic polynomial} of a graph~$G$ is defined
as~$p(G, x) = \mathrm{det}(x I - A_G)$, where~$A_G$ denotes the
adjacency matrix of~$G$ and~$I$ denotes the identity matrix of the same dimension as~$A_G$.
Here,~$\mathrm{det}$ denotes the determinant of a square matrix.
The zeros of~$p(G, x)$ are the eigenvalues of the adjacency matrix~$A_G$.
Two graphs~$G$ and~$H$ are called \emph{cospectral}
if their characteristic polynomials are the same; equivalently,~$G$
and~$H$ are cospectral if their adjacency matrices~$A_G$ and~$A_H$
have the same multisets of eigenvalues. Dell, Grohe and Rattan proved
the following characterization of cospectral graphs; see Proposition~9
in~\cite{DBLP:conf/icalp/DellGR18}, attributed
to~\cite{VanDamHaemers2003}.

\begin{theorem} \label{thm:grohe-cospectral} \cite{DBLP:conf/icalp/DellGR18} For every two
  graphs~$G$ and~$H$, we have that~$G$ and~$H$ are cospectral if and
  only if~$\hom(\class{C},G) = \hom(\class{C},H)$, where~$\class{C}$
  is the class of all cycles, including the degenerate cycles with one
  and two vertices.
\end{theorem}

Here we show that, in contrast, cospectrality cannot be characterized
as the restriction of a right profile. Interestingly, as
we will see, the proof of this result is an application of
Theorem~\ref{thm:general} with~$k = 3$.

\begin{theorem}
  There is no class~$\class{F}$ of graphs such that for every two
  graphs~$G$ and~$H$, we have  that~$G$ and~$H$ are cospectral if and
  only if~$\hom(G,\class{F}) = \hom(H,\class{F})$.
\end{theorem}

\begin{proof} It is known that any two~$\equiv^3_{\Clogic}$-equivalent
  graphs are cospectral; since all cycles have treewidth at most two,
  this follows by combining Theorems~\ref{thm:grohe-cospectral}
  and~\ref{thm:grohe-ck-equiv}, but see also Theorem~2.1
  in~\cite{DawarSeveriniZapata2017}. Moreover, any two cospectral
  graphs have the same number of vertices. Therefore,
  cospectrality is finer than~$\equiv^1_\Clogic$ and coarser
  than~$\equiv^3_{\Clogic}$. The result follows from
  Theorem~\ref{thm:general}.
\end{proof}

\section{Left Profiles} \label{sec:left}
For every graph $G$, there is a polynomial $\chrom(G, x)$ in the variable $x$, called the \emph{chromatic polynomial} of $G$, such that, for every $k\geq 1$, the value $\chrom(G,k)$ is the number of $k$-colorings of $G$ (see the survey \cite{read1968introduction} for the history of chromatic polynomials and basic facts about them).
Furthermore, if $G$ has $n$ vertices, then $\chrom(G, x)$ has degree $n$ and its leading coefficient is $1$. For example,
\begin{itemize}
\item
$\chrom(I_n,x)=x^n$, where $n\geq 1$;
\item $\chrom(K_n,x ) = x(x-1)\cdots (x-n+1)$, where $n\geq 1$;
\item  $\chrom(C_n,x) = (x-1)^n + (-1)^n(x-1)$, where  $n\geq 3$.
\end{itemize}

\ignore{
The \emph{chromatic polynomial} $\chrom(G, x)$ of a graph $G$ is a graph polynomial in the variable $x$ that gives the number of $n$-colorings ($n \geq 1$) of $G$ by evaluating $\chrom(G, n)$.

\begin{example}\label{chrom_eq_example}
Suppose that we have $x$ colors available to draw the vertices in a graph so that no two adjacent vertices are drawn the same color, we will show below how to derive $\chrom(G, x)$ for a graph $G$, namely the number of $x$-colorings of $G$. Notice that $\chrom(G, n) = 0$ if $n < \card{\vertex(G)}$, which means that $\chrom(G, x)$ consistently gives the number of $x$-colorings of $G$.

Let us first consider $\indep[1]$. There are $x$ colors to draw the single vertex, so $\chrom(\indep[1], x) = x$. On other hand, since $\clique[1] = \indep[1]$, we have $\chrom(\clique[1], x) = \chrom(\indep[1], x) = x$.

Next, we consider $\clique[2]$. There are $x$ colors to draw a vertex and there are $x - 1$ colors to draw a different vertex, so $\chrom(\clique[2], x) = x(x - 1)$. In $\clique[n]$ for $n \geq 3$, we have $x$ colors available for a vertex $v$, $x - 1$ colors available for a different vertex $v'$, $x - 2$ colors available for a vertex $v''$ different than $v$ and $v''$, and so on. Therefore $\chrom(\clique[n], x)$ is $x(x - 1) \etc (x - n + 1)$.
}

There are two useful techniques for deriving the chromatic polynomial $\chrom(G, x)$ of an arbitrary graph $G$.

\smallskip

\noindent \emph{Multiplicativity}. If~$G = G_1 \dunion G_2$ is the disjoint union of two graphs~$G_1$
and~$G_2$, then~$\chrom(G, x) = \chrom(G_1, x) \cdot \chrom(G_2, x)$.

\smallskip

\noindent  \emph{Addition-Contraction Recursion}. If~$u, v \in \vertex(G)$ are distinct and~$(u, v) \notin \edge(G)$, then $\chrom(G, x) = \chrom(G_1, x) + \chrom(G_2, x)$ where~$G_1$ is obtained from~$G$ by adding the edge~$(u,v)$, while~$G_2$ is obtained from~$G$ by contracting the two vertices~$u, v$.

    \ignore{
    \footnote{In fact, we are using an equivalent form of deletion-contraction recursion where on the righthand-side of the formula we add an edge to $G$ rather than deleting an edge from it to obtain $G'$.}
    }
    The base cases of the addition-contraction recursion are of the form $\chrom(\clique[n], x) = x(x - 1) \cdots (x - n + 1)$.
     Thus, to obtain $\chrom(G, x)$ for a graph $G$ that is not a clique and has $n$ vertices, we can use this recursion to expand $\chrom(G, x)$ until every term  in the expansion is of the form $\chrom(\clique[m], x)$ with $m\leq n$.
%


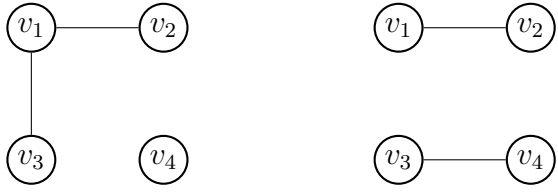
\begin{figure}[!t]
\centering
\begin{tikzpicture}[every node/.style={draw=black,thick,circle,inner sep=0pt}]
\node[circle,inner sep=2pt,minimum size=5pt] (a1) {$v_1$};
\node[circle,inner sep=2pt,minimum size=5pt, right of=a1] (a2) {$v_2$};
\node[circle,inner sep=2pt,minimum size=5pt, below of=a1] (a3) {$v_3$};
\node[circle,inner sep=2pt,minimum size=5pt, below of=a2] (a4) {$v_4$};
\node[circle,inner sep=2pt,minimum size=5pt, right =70pt of a2] (a5) {$v_1$};
\node[circle,inner sep=2pt,minimum size=5pt, right of=a5] (a6) {$v_2$};
\node[circle,inner sep=2pt,minimum size=5pt, below of=a5] (a7) {$v_3$};
\node[circle,inner sep=2pt,minimum size=5pt, below of=a6] (a8) {$v_4$};
\draw
(a2) edge (a1)
(a3) edge (a1)
(a6) edge (a5)
(a8) edge (a7)
;
\end{tikzpicture}
\caption{Two chromatically equivalent graphs, $X_1$ (left) and $X_2$ (right).}
\label{fig_chrom_eq}
\end{figure}
\begin{example}\label{chrom_eq_example}
Consider the two graphs $X_1 = \indep[1] \dunion \pathgraph[3]$ and $X_2 = \pathgraph[2] \dunion \pathgraph[2]$ depicted in Figure~\ref{fig_chrom_eq}. Using the preceding two techniques and the chromatic polynomials of $\clique[2]$ and $\clique[3]$, it is easy to verify that
$\chrom(X_1, x) = x^2(x - 1)^2 = \chrom(X_2, x)$.
\end{example}

We say that two graphs $G$ and $H$ are \emph{chromatically equivalent} if they have the same number of $n$-colorings, for every $n\geq 1$. This is equivalent to saying that $G$ and $H$ have the same chromatic polynomial.
 The two graphs $X_1, X_2$  in Example \ref{chrom_eq_example} are chromatically equivalent,  a fact that will be used in what follows.
   Furthermore, all trees with the same number of vertices are pairwise chromatically equivalent, because if $T$ is a tree
with $n$ vertices, then $\chrom(T,x) = x(x-1)^{n-1}$.

It is obvious  that chromatic equivalence is characterized by restricting the right profile to the class $\class{K}$ of all cliques, that is, two graphs $G$ and $H$ are chromatically equivalent if and only if $\hom(G, \class{K})= \hom(H, \class{K})$.
In contrast, we will show that chromatic equivalence cannot be characterized by  restricting the left profile.

\begin{theorem}\label{lovasz_vector_not_characterize_chrom_eq}
There is no class $\class{F}$ of graphs such that for every two graphs $G$ and $H$, we have that $G$ and $H$ are chromatically equivalent if and only if $\lefthomvec(\class{F}, G) = \lefthomvec(\class{F}, H)$.
\end{theorem}

Before proving Theorem \ref{lovasz_vector_not_characterize_chrom_eq}, we present two lemmas.

\begin{lemma}
\label{chr_equv_prop}
For every $2$-colorable connected graph $G$ with at least three vertices, we have that $\sur(G, \pathgraph[3]) > 0$.
\end{lemma}

\begin{proof}
Let $G$ be a $2$-colorable connected graph with $\card{\vertex(G)} \geq 3$. We argue that $\sur(G, \pathgraph[3]) > 0$.

For simplicity, in what follows we assume that $\clique[2]$ is a graph with vertex set $\vertex(\clique[2]) = \sete{v_1, v_2}$ and edge set $\edge(\clique[2]) = \sete{(v_1, v_2)}$, and that $\pathgraph[3]$ is a graph with $\vertex(\pathgraph[3]) = \sete{v_1, v_2, v_3}$ and $\edge(\pathgraph[3]) = \sete{(v_1, v_2), (v_2, v_3)}$; in other words, $\clique[2]$ is a subgraph of $\pathgraph[3]$.

Assume now that  $\card{\vertex(G)} \geq 2$. Then $G$ must contain an edge because it is connected. Moreover, we have $\sur(G, \clique[2]) > 0$ as $G$ is $2$-colorable. Let $h: G \to \clique[2]$ be a surjective homomorphism and, without loss of generality, assume that $\card{\inv{h}(v_1)} > 1$, i.e., more than one vertex in $\vertex(G)$ is mapped to $v_1 \in \vertex(\clique[2])$ under $h$ (note that at least one of $\inv{h}(v_1)$ and $\inv{h}(v_2)$ has cardinality $> 1$). Pick two distinct vertices $a_1, a_2 \in \vertex(G)$ such that $h(a_1) = h(a_2) = v_1$; then there are two vertices $b_1, b_2 \in \vertex(G)$ (not necessarily distinct) such that $(a_1, b_1), (a_2, b_2) \in \edge(G)$ (since $G$ is connected) and $h(b_1) = h(b_2) = v_2$. It follows that the homomorphism $h': G \to \pathgraph[3]$ such that $h'(u) = h(u)$ for $u \neq a_2$ and $h'(a_2) = v_3$ is surjective. We conclude that $\sur(G, \pathgraph[3]) > 0$.
\end{proof}

Given any graph $G$, we define the \emph{$n$-fold disjoint union $G^{\oplus n}$ of $G$}, where $n \geq 0$, inductively as follows:
\begin{displaymath}
\begin{array}{llllll}
G^{\oplus 0} & \defas & \emptyset; &
G^{\oplus n + 1} & \defas & G^{\oplus n} \dunion G.\cr
\end{array}
\end{displaymath}
We also set $\hom(\emptyset, H) \defas 1$ for every graph $H$.

The next lemma  tells which graphs $G$ distinguish the graphs $X_1$ and $X_2$ in Example \ref{chrom_eq_example} via $\hom(G, X_1)$ and $\hom(G, X_2)$.

\begin{lemma}
\label{chr_equv_key_lemma} If $G$ is a graph, then
$\hom(G, X_1) \geq \hom(G, X_2)$. Furthermore, $\hom(G, X_1) = \hom(G, X_2)$ holds if and only if either $G$ is not $2$-colorable or
$G$ is a disjoint union of vertices and edges (in symbols,  $G = \indep[m] \dunion \pathgraph[2]^{\oplus n}$ for some $m, n \geq 0$ such that $m + n \geq 1$).
\end{lemma}
\begin{proof} Let $G$ be a graph. We distinguish two cases.

\smallskip

\noindent\emph{Case 1:} $G$ is connected.
If $G$ is not $2$-colorable, then $\hom(G,X_1) = 0 = \hom(G,X_2)$, since both $X_1$ and $X_2$ are $2$-colorable. If $\card{\vertex(G)} \leq 2$, then, since $G$ is connected, we have that $G= I_1$ or $G=P_2$, from which it follows easily that $\hom(G,X_1) = 4 = \hom(G,X_2)$.
So, assume that $G$ is  $2$-colorable and  $\card{\vertex(G)} \geq 3$. We will show that, in this case, $\hom(G,X_1) > \hom(G,X_2)$.
Since $G$ is connected, it must have at least one edge.
 Then any homomorphic image of $G$ must contain an edge and must be connected as well since the homomorphic image of a connected graph is connected. Thus, the calculation of $\hom(G, X_1)$ and of $\hom(G, X_2)$ can be broken down to the calculation of the number of surjective homomorphisms from $G$ onto a connected subgraph of $X_1$ or $X_2$ that contains at least one edge; note that such a\ignore{connected} subgraph must be isomorphic to $\pathgraph[2]$ or $\pathgraph[3]$. Since $X_1$ has exactly two subgraphs isomorphic to $\pathgraph[2]$ and one subgraph isomorphic to $\pathgraph[3]$ whereas $X_2$ has exactly two subgraphs isomorphic to $\pathgraph[2]$ and no subgraph isomorphic to $\pathgraph[3]$, we obtain
 $\hom(G, X_1) =  2 \sur(G, \pathgraph[2])  + \sur(G, \pathgraph[3])$ and $\hom(G,X_2) = 2 \sur(G, \pathgraph[2])$. By Lemma \ref{chr_equv_prop}, we have  that $\sur(G, \pathgraph[3]) > 0$. It follows that $\hom(G, X_1) > \hom(G,X_2)$.

 \ignore{
\begin{displaymath}
\begin{array}{llll}
\hom(G, X_1) & = & 2 \sur(G, \pathgraph[2])  + \sur(G, \pathgraph[3]) & \cr
>  & 2 \sur(G, \pathgraph[2])  & = & \hom(G,X_2), \cr
\end{array}
\end{displaymath}
where the inequality holds because, by Lemma \ref{chr_equv_prop}, we have  that $\sur(G, \pathgraph[3])$.
}

Thus, if $G$ is connected, then  $\hom(G, X_1) \geq \hom(G, X_2)$, and the equality holds if and only if $G$ is not $2$-colorable or $G \in \sete{\indep[1], \pathgraph[2]}$.

\smallskip

\noindent\emph{Case 2:} $G$ is not connected.
Let $H_1,\ldots,H_k$, $k\geq 2$, be the connected components of $G$. Thus, $G = \bdunion^k_{i = 1} H_i$ and so
 $\hom(G, X_1)  =  \ds{\Pi^k_{i = 1} \hom(H_i, X_1)}$ and $\hom(G,X_2) =  \ds{\Pi^k_{i = 1} \hom(H_i, X_2)}$. Since each $H_i$ is  connected, the previous case implies that $\hom(H_i, X_1) \geq \hom(H_i, X_2)$, for $1\leq i\leq k$. Consequently, $\hom(G,X_1) \geq \hom(G,X_2)$. Furthermore, the equality holds if and only if $G$ is not $2$-colorable (i.e., at least one $H_i$ is not $2$-colorable) or every $H_i\in \sete{\indep[1], \pathgraph[2]}$, which means that $G$ is not $2$-colorable or  $G = \indep[m] \dunion \pathgraph[2]^{\oplus k - m}$, for some $m$ with $0 \leq m \leq k$.
\ignore{
\begin{displaymath}
\begin{array}{llll}
\ & \ds{\hom(\bdunion^k_{i = 1} H_i, X_1)} & = & \ds{\prod^k_{i = 1} \hom(H_i, X_1)} \cr
\geq & \ds{\prod^k_{i = 1} \hom(H_i, X_2)} & = & \ds{\hom(\bdunion^k_{i = 1} H_i, X_2)} \cr
\end{array}
\end{displaymath}
where the inequality follows by the previous case and so the equality holds if and only if for each $1 \leq i \leq k$, $\chrom(H_i) \geq 3$ or $H_i \in \sete{\indep[1], \pathgraph[2]}$. Hence, if the equality holds then either $\chrom(G) \geq 3$ or $G = \indep[m] \dunion (\pathgraph[2])^{k - m}$ for some $0 \leq m \leq k$.
}
\end{proof}

We are now ready to prove Theorem \ref{lovasz_vector_not_characterize_chrom_eq}.

\begin{proof}[Proof of Theorem \ref{lovasz_vector_not_characterize_chrom_eq}]
Towards a contradiction, suppose that there is a class $\class{F}$ of graphs such that for every two graphs $G$ and $H$, we have that $G$ and $H$ are chromatically equivalent if and only $\lefthomvec(\class{F}, G) = \lefthomvec(\class{F}, H)$.
In particular, for the chromatically equivalent graphs $X_1, X_2$  in Example \ref{chrom_eq_example}, we have that $\hom(F, X_1) = \hom(F, X_2)$, for every $F \in \class{F}$. Lemma \ref{chr_equv_key_lemma} implies that $\class{F} \subseteq \class{F}_1 \union \class{F}_2$, where
\begin{displaymath}
\begin{array}{lll}
\class{F}_1 & = & \{G ~|~ \mbox{$G$ is not $2$-colorable} \} \cr
\class{F}_2 & = & \sett{\indep[m] \dunion \pathgraph[2]^{\oplus n}}{\mathmode{m, n \geq 0} and \mathmode{m + n \geq 1}}. \cr
\end{array}
\end{displaymath}

Let $G = \cycle[8]$ be the $8$-cycle and let  $H = \cycle[4]\oplus \cycle[4]$ be the disjoint union of two $4$-cycles.
Then $G$ and $H$ are not chromatically equivalent because they have different chromatic polynomials. Indeed,
$\chrom(G, x)  =   (x - 1)^8 + (x - 1)$, while $\chrom(H, x) = ( (x - 1)^4 + (x - 1))^2$.
However, no graph $F$ in $\class{F}_1 \union \class{F}_2$ distinguishes $G$ and $H$ via $\hom(F, G)$ and $\hom(F, H)$. The reason is that if $F$ is not $2$-colorable, then $\hom(F, G) = 0 =  \hom(F, H)$, while if $F = \indep[m] \dunion \pathgraph[2]^{\oplus n}$ for some $m, n \geq 0$ with $m + n \geq 1$, then $\hom(F, G) = 8^m 16^n = \hom(F, H) $.
\end{proof}

\ignore{
\begin{displaymath}
\begin{array}{lrll}
\    & \chrom(G, x) & = & (x - 1) + (x - 1)^8 \cr
\neq & ((x - 1) + (x - 1)^4)^2 & = & \chrom(H, x). \cr
\end{array}
\end{displaymath}
}

\section{Left and Right Profiles} \label{sec:left-right}
In the previous two sections, we saw that several natural relaxations of isomorphism can be captured by restricting one of the two (the left or the right) profiles, but not by the other. In this section, we shall see that other natural, logic-based relaxations of isomorphism  cannot be captured by restricting either of the two profiles.

Let $k\geq 2$ be a positive integer. Recall that $\folog^k$ is the fragment of first-order logic consisting of all FO-formulas with at most $k$ distinct variables.
Two graphs $G$ and $H$ are said to be $\folog^k$-equivalent, denoted by  $G \equiv^{k}_{\mathrm{FO}} H$, if $G$ and $H$ satisfy the same $\folog^k$-sentences. The finite-variable logic $\folog^k$, $k\geq 2$, have been studied extensively in finite model theory (see \cite{DBLP:books/sp/Libkin04,DBLP:series/txtcs/GradelKLMSVVW07}); in particular, it is well known that $\folog^k$-equivalence can be characterized in terms of the $k$-pebble game, $k\geq 2$.
The next result tells that $\folog^k$-equivalence cannot be captured by restricting the left profile or the right profile.

\begin{proposition}
\label{prop:FO-k} Consider a positive integer $k\geq 2$.
\begin{enumerate}
\item
There is no class $\class{F}$ of graphs such that for every two graphs $G$ and $H$, we have that
 $G \equiv^{k}_{\mathrm{FO}} H$ if and only if $\lefthomvec(\class{F}, G) = \lefthomvec(\class{F}, H)$.

\item There is no class $\class{F}$ of graphs such that for every two graphs $G$ and $H$, we have that
 $G \equiv^{k}_{\mathrm{FO}} H$ if and only if $\righthomvec(G, \class{F}) = \righthomvec(H, \class{F})$.
\end{enumerate}
\end{proposition}
\begin{proof}
For the first part, suppose that such a class $\class{F}$ exists. Pick a graph  $D$  in $\class{F}$ and let $c$ be the number of the vertices of $D$. Let  $m= \max(c, k)$ and
consider the cliques $\clique[m]$ and $\clique[m + 1]$ with $\vertex(\clique[m]) = \sete{v_1, \etc, v_m}$ and $\vertex(\clique[m + 1]) = \sete{v_1, \etc, v_m, v_{m + 1}}$. Clearly,  $\clique[m] \equiv_{\folog}^m \clique[m + 1]$, hence $\clique[m] \equiv_{\folog}^k \clique[m + 1]$ because $k \leq m$. By the assumption about the class $\class{F}$, we have  $\lefthomvec(\class{F}, \clique[m]) = \lefthomvec(\class{F}, \clique[m + 1])$; in particular, we have   $\hom(D, \clique[m]) = \hom(D, \clique[m + 1])$.
Furthermore, $\hom(D, \clique[m]) > 0$, since $m \geq c = \card{\vertex(D)}$ and there is an injective homomorphism from $D$ to $\clique[m]$. Therefore, $\hom(D, \clique[m]) = \hom(D, \clique[m + 1]) > 0$.
This, however, is a contradiction, since $\hom(D, \clique[m + 1]) > \hom(D, \clique[m])$.  Indeed, first note that $\Hom(D, \clique[m]) \subseteq \Hom(D, \clique[m + 1])$. Next, given a homomorphism $h: D \to \clique[m]$, we take the least-indexed vertex in the image $h(D)$, say $v_r$ (note that $ r \leq m$), and substitute $v_{m + 1}$ in $h$ for $v_r$ to obtain a homomorphism $h': D \to \clique[m + 1]$ such that for every $u \in \vertex(D)$, we have that $h'(u) = v_{m + 1}$ if $u \in \inv{h}(v_r)$, and $h'(u) = h(u)$ if $u \in \vertex(D) \setminus \inv{h}(v_r)$.
\ignore{
\begin{displaymath}
h'(u) =
\begin{cases}
v_{m + 1} & \mbox{if \mathmode{u \in \inv{h}(v_r)}}\cr
h(u) & \mbox{if \mathmode{u \in \vertex(D) \setminus \inv{h}(v_r)}}.\cr
\end{cases}
\end{displaymath}}
Obviously, $h' \notin \hom(D, \clique[m])$, which proves the first part.

For the second part, suppose that such a class $\class{F}$ exists. It cannot be the case $\class{F} = \sete{\indep[1]}$. Otherwise, we would have $\righthomvec(\indep[1], \class{F}) = \righthomvec(\indep[2], \class{F})$ which, by assumption, would imply $\indep[1] \equiv_{\folog}^k \indep[2]$; this is absurd because $\indep[1] \satis \forall x \forall y \, x = y$, while $\indep[2] \not \satis \forall x \forall y \, x = y$.
Therefore, $\class{F}$ contains a graph $D$ with at least two vertices.
Let $c =\card{\vertex(D)} \geq 2$ and let $m = \max(c, k)$. Consider two independent sets, $\indep[m]$ and $\indep[m + 1]$. Obviously, $\indep[m] \equiv_{\folog}^m \indep[m + 1]$ and hence $\indep[m] \equiv_{\folog}^k \indep[m + 1]$ since $m \geq k$. The assumption about $\class{F}$ implies that  $\righthomvec(\indep[m], \class{F}) = \righthomvec(\indep[m + 1], \class{F})$. Therefore, $\hom(\indep[m], D) = \hom(\indep[m + 1], D)$, which is a contradiction because $c = \card{\vertex(D)} \geq 2$ and $\hom(\indep[m], D) = c^m < c^{m + 1} = \hom(\indep[m + 1], D)$.
\end{proof}

The \emph{quantifier-depth} of a first-order formula is a positive integer that measures the nesting of quantifiers in that formula; it gives rise to a parametrization of first-order logic that is different from the parametrization according to the number of distinct formulas.
Let $k$ be a positive integer and let $G$ and $H$ be two graphs. We write $G \equiv^k_{{\mathrm QR}} H$ to denote that $G$ and $H$ satisfy the same FO-sentences of quantifier rank at most $k$. It is well known that the equivalence relation $\equiv^k_{{\mathrm QR}}$ is characterized in terms of the $k$-move Ehrenheufcht-Fra\"{i}ss\'e game (see \cite{DBLP:books/sp/Libkin04,DBLP:series/txtcs/GradelKLMSVVW07}).
The next result tells that $\equiv^k_{{\mathrm QR}}$ cannot be captured by restricting the left profile or the right profile. The proof is omitted because it is essentially the same as that of Proposition \ref{prop:FO-k}.

\begin{proposition}\label{no_vectors_characterize_folog_qd_k}
 Consider a positive integer $k\geq 2$.
\begin{enumerate}
\item
There is no class $\class{F}$ of graphs such that for every two graphs $G$ and $H$, we have that
 $G \equiv^{k}_{\mathrm{QR}} H$ if and only if $\lefthomvec(\class{F}, G) = \lefthomvec(\class{F}, H)$.

\item There is no class $\class{F}$ of graphs such that for every two graphs $G$ and $H$, we have that
 $G \equiv^{k}_{\mathrm{QR}} H$ if and only if $\righthomvec(G, \class{F}) = \righthomvec(H, \class{F})$.
\end{enumerate}
\end{proposition}

\section{Extensions and Discussion} \label{sec:extensions}
Up to this point, we have focused on graphs that have no self-loops,
no multi-edges, and no weights on the vertices or the edges;
furthermore, the computation of the homomorphism count uses integer
arithmetic only. In this section, we discuss two extensions of the
framework: the first involves graphs with loops and real numbers as
weights on the vertices and the edges, while the second involves
defining the homomorphism count over arbitrary semirings.

\subsection{Extension to Graphs with Loops and Real  Weights}
There is substantial literature on extended notions of homomorphism
counts that involve \emph{weighted} graphs, that is, undirected graphs
with self-loops and weights on each vertex and each edge (including
self-loops), but no multi-edges. In particular, Lov\'asz defines the
notion of the homomorphism count~$\hom(G,H)$, where~$G$ is an ordinary
graph (no self-loops, no multi-edges, no weights) and~$H$ is a
weighted graph \cite{DBLP:books/daglib/0031021}. Specifically, assume
that~$H=(V(H), E(H), w)$ is a weighted graph, where~$w$ is a
real-valued weight function defined on~$V(H)\cup
E(H)$. If~$G=(V(G),E(G))$ is a graph, then a \emph{homomorphism}
from~$G$ to~$H$ is a homomorphism from~$G$ to the (unweighted) looped
graph~$H'=(V(H),E(H))$ (note that here an edge of~$G$ can be mapped to
a self-loop of~$H'$). Let~$\Hom(G,H')$ denote the set of all
homomorphisms from~$G$ to~$H'$ and define the homomorphism
count~$\hom(G,H)$ as
\begin{equation*}
\hom(G,H) = \displaystyle\sum_{h \in \Hom(G,H')} \displaystyle\prod_{u\in V(G)}w(h(u))\displaystyle\prod_{e \in E(G)}w(h(e))
\end{equation*}
The earlier definition of~$\hom(G,H)$ between two graphs~$G$ and~$H$
is the special case of this in which~$H$ is turned into a weighted
graph with weight~$1$ on each vertex and each edge.

Several important graph invariants can be expressed as graph
polynomials. In Section \ref{sec:left}, we encountered such a graph
polynomial, namely, the chromatic polynomial~$\chrom(G,k)$, which
gives the number of the~$k$-colorings of~$G$. It is known that several
other fundamental graph polynomials can be expressed using the
preceding extended notion of homomorphism counts (for an overview, see
\cite[Section 5.3]{DBLP:books/daglib/0031021}). Here, we discuss two
such polynomials.

The \emph{cluster expansion polynomial}~$\textsc{cep}(G;x,y)$ of a
graph~$G$ is a bivariate polynomial that, among other things,
generalizes the chromatic polynomial. If~$G$ is a graph, then, by
definition,
$$
\textsc{cep}(G;x,y) = \displaystyle\sum_{A\subseteq
  E(G)}x^{c(A)}y^{|A|},
$$
where~$c(A)$ is the number of connected components of the
graph~$(V(G),A)$ and~$|A|$ is the cardinality of the set~$A$.  It can
be shown that~$\textsc{cep}(G;k,-1)=\chrom(G,k)$. More importantly,
the cluster expansion polynomial is a version of the Tutte polynomial
\cite{tutte2004graph}, which is arguably the most fundamental graph
polynomial as it encapsulates a great deal of information about the
graph with which it is associated.

The cluster expansion polynomial can be expressed in terms of the
homomorphism count between graphs and certain weighted graphs (see
\cite[Section 5.3]{DBLP:books/daglib/0031021})). Specifically, for
every~$k\geq 1$ and every real~$y$, let~$K_{k,y}$ be the clique on~$k$
vertices with a self-loop added at every vertex and with the following
weights: every vertex has weight~$1$, every self-loop has
weight~$1+y$, and every other edge has weight~$1$. Then it can be
proved that~$\textsc{cep}(G;k,y)= \hom(G,K_{k,y})$. It follows that
if~$\class{F}$ is the class of all weighted graphs of the
form~$K_{k,y}$, then the right profile restricted
to~$\class{F}$ captures the equivalence relation ``the graphs~$G$
and~$H$ have the same cluster expansion polynomial".

The \emph{independence polynomial}~$I(G;x,y)$ of a graph~$G$ is a
bivariate polynomial that encapsulates information about the
independent sets of~$G$. If~$G$ is a graph, then let~$\mathscr{I}(G)$
denote the collection of its independent sets, and define
$$I(G; x,y) = \displaystyle\sum_{U \in \mathscr{I}(G)} x^{|U|} y^{|V(G)\setminus U|}.$$
Note that
$$I(G;x,1) = \displaystyle \sum _{U \in \mathscr{I}(G)}x^{|U|},$$
which is the univariate independence polynomial introduced by Gutman
and Harary \cite{gutman1983generalizations}.  Like the cluster
expansion polynomial, the independence polynomial can be expressed in
terms of the homomorphism count between graphs and certain weighted
graphs (see \cite[Section
4.2]{DBLP:journals/ejc/GarijoGN11}). Specifically,
let~$L =(V(L),E(L))$ be the ``lollipop" graph, i.e., the graph with
two vertices~$a$,~$b$, an edge~$(a,b)$ and a self-loop~$(b,b)$. For
every two real numbers~$x$ and~$y$, let~$L_{x,y}$ be the weighted
graph obtained from~$L$ by putting~$x$ as the weight of the
vertex~$a$, putting~$y$ as the weight of the vertex~$b$, and
putting~$1$ as the weight of the edge~$(a,b)$ and the
self-loop~$(b,b)$. Then it can be proved
that~$I(G;x,y) = \hom(G,L_{x,y})$.  It follows that if~$\class{F}$ is
the class of all weighted graphs of the form~$L_{x,y}$, then the
right profile restricted to~$\class{F}$ captures the
equivalence relation ``the graphs~$G$ and~$H$ have the same
independence polynomial".

The preceding discussion suggests that the investigation we embarked
on here should be expanded to an investigation of the expressive power
of profiles restricted to classes of weighted graphs, as
such vectors give rise to a variety of equivalence relations arising
from graph polynomials.

\subsection{Extension to Arbitrary Semirings}

Let~$G$ be a graph and let~$H$ be a weighted graph. The expression
defining the homomorphism count~$\hom(G,H)$ is a sum of products of
real numbers. This sum of products is also meaningful over an
arbitrary semiring~$\mathbb{K}= (K, +, \times, 0, 1)$, where~$+$
and~$\times$ are the addition and multiplication operations on~$K$,
and~$0$ and~$1$ are the identity elements of~$+$ and~$\times$. Thus,
we can define the homomorphism count~$\hom_{\mathbb K}(G,H)$ for a
graph~$G$ and a~$\mathbb K$-weighted graph~$H$, where the weight
function takes values in the universe~$K$ of an arbitrary, but fixed,
semiring~$\mathbb{K}$. For example, the standard homomorphism
count~$\hom(G,H)$, where~$G$ and~$H$ are graphs, coincides with the
homomorphism count~$\hom_{\mathbb N}(G,H)$,
where~$\mathbb{N}= (\{0,1,2,\ldots\},+, \times, 0, 1)$ is the
\emph{bag} semiring of the non-negative integers, and~$H$ is viewed as
a~${\mathbb N}$-weighted graph with weight~$1$ on each vertex and each
edge.  And, of course, the discussion in the preceding section is
about the homomorphism count~$\hom_{\mathbb R}(G,H)$,
where~$\mathbb{R}=(R,+,\times,0,1)$ is the semiring (actually, the
field) of the real numbers.

These considerations pave the way for a further expansion of the
framework to homomorphism counts with respect to some arbitrary, but
fixed, semiring.  As a concrete case in point, consider the
\emph{Boolean} semiring~$\mathbb{B} = (\{0,1\},\vee,\wedge,0,1)$,
which has disjunction~$\vee$ and conjunction~$\wedge$ as operations,
and~$0$ (false) and~$1$ (true) as the identity elements of~$\vee$
and~$\wedge$. Let~$G$ be a graph and let~$H$ be a~$\mathbb B$-graph
with weight~$1$ on each vertex and each edge.
Then~$\hom_{\mathbb B}(G,H)$ is the \emph{sign} of the standard
homomorphism count~$\hom(G,H)$ indicating the existence or
non-existence of a homomorphism from~$G$ to~$H$, that
is,~$\hom_{\mathbb B}(G,H)= 1$ if there is a homomorphism from~$G$
to~$H$, and~$\hom_{\mathbb B}(G,H) =0$, otherwise.

Let $G$ be a graph and let $\class{F}$ be a class of graphs. Put $\hom_{\mathbb B}(\class{F}, G)= \seq{\hom_{\mathbb B}(D, G)}{D \in \class{F}}$ and $\hom_{\mathbb B}(G,\class{F})= \seq{\hom_{\mathbb B}(G, D)}{D \in \class{F}}$ for the left and the right profiles of $G$ restricted to $\class{F}$.

Recall that $\classofgraphs$ is the class of all graphs. Obviously, for every
graph, there is a homomorphism from that graph to itself. Therefore,
for every two graphs~$G$ and~$H$, we have that the following
statements are equivalent:
\begin{enumerate}
\item $\hom_{\mathbb B}(\class{G}, G)= \hom_{\mathbb B}(\class{G}, H)$.
\item $\hom_{\mathbb B}(G,\class{G})= \hom_{\mathbb B}(H,\class{G})$.
\item $G$ and $H$ are homomorphically equivalent.
\end{enumerate}
By definition,~$G$ and~$H$ are \emph{homomorphically equivalent} if
there are homomorphisms from~$G$ to~$H$, and from~$H$ to~$G$.  The
notion of homomorphic equivalence plays an important role in several
different areas, including database theory and constraint
satisfaction. Furthermore, homomorphic equivalence coincides with
isomorphism on graphs that are cores (see
\cite{DBLP:books/daglib/0013017} for detailed information about these
notions).

The \emph{chromatic number}~$\chrom(G)$ of a graph~$G$ is the smallest
positive integer~$k$ such that~$G$ has a~$k$-coloring. Clearly,~$G$
has a~$k$-coloring if and only if~$\hom(G, \clique[k]) > 0$;
furthermore, if~$\hom(G, \clique[k]) > 0$,
then~$\hom(G, \clique[m]) > 0$, for all~$m > k$.  Thus,
two graphs~$G$ and~$H$ have the same chromatic number if and only
if~$\hom_{\mathbb B}(G,\classofcliques)= \hom_{\mathbb
  B}(H,\classofcliques)$.  In contrast, the next result asserts that
the equivalence relation of two graphs having the same chromatic
number cannot be captured by restricting the left profile~$\hom_{\mathbb B}(\class{G}, G)$.

\begin{proposition} \label{prop:chromatic} There is no
  class~$\class{F}$ of graphs such that for every two graphs~$G$
  and~$H$, we have that~$\chrom(G) = \chrom(H)$ if and only
  if~$\hom_{\mathbb B}(\class{F}, G) = \hom_{\mathbb B}(\class{F},
  H)$.
\end{proposition}
\begin{proof}
Towards a contradiction, assume such a class~$\class{F}$ exists.

Consider the cycles $\cycle[2k]$ of even length, $k \geq 2$. They
  have chromatic number~$2$,
  hence $\hom_{\mathbb B}(\class{F}, \cycle[2k]) = \hom_{\mathbb
    B}(\class{F}, \cycle[2m])$, for all~$k, m\geq 2$.  Next, consider
  the cycles~$\cycle[2k + 1]$ of odd length,~$k \geq 1$. They have
  chromatic number~$3$,
  hence~$\hom_{\mathbb B}(\class{F}, \cycle[2k + 1]) = \hom_{\mathbb
    B}(\class{F}, \cycle[2m + 1])$, for all~$k, m \geq 1$.  We will
  show
  that~$\hom_{\mathbb B}(\class{F}, \cycle[2k]) = \hom_{\mathbb
    B}(\class{F}, \cycle[2k + 1])$, for~$k \geq 2$, which will be a
  contradiction, since~$\cycle[2k]$ and~$\cycle[2k + 1]$ have
  different chromatic numbers.

  Fix a~$k \geq 2$ and let~$D$ be an arbitrary graph in~$\class{F}$.

  \emph{Case 1.}~$D$ is a~$2$-colorable graph. In this case, we have
  that~$\hom_{\mathbb B}(D, \cycle[2k]) = 1 = \hom_{\mathbb B}(D,
  \cycle[2k + 1])$.

  \emph{Case 2.}~$D$ is a non-$2$-colorable graph. It follows that~$D$
  contains an odd cycle~$\cycle[2l + 1]$, for some~$l \geq 1$. Pick a
  cycle~$\cycle[2n]$ with~$2n > 2l + 1$.
  Then~$\hom_{\mathbb B}(D, \cycle[2n]) = 0$, since homomorphisms map
  odd cycles to odd cycles of smaller or equal length,
  but~$\cycle[2n]$ contains no such cycle.
  Therefore,~$\hom_{\mathbb B}(D, \cycle[2k]) = 0$ as well.  Next,
  pick a cycle~$\cycle[2n + 1]$ with~$n > l$. By the same
  reasoning,~$\hom_{\mathbb B}(D, \cycle[2n + 1]) = 0$, and
  so~$\hom_{\mathbb B}(D, \cycle[2k + 1]) = 0$ (recall that $\chrom(\cycle[2n + 1]) = \chrom(\cycle[2k + 1]) = 3$).

  From the preceding case analysis, we  conclude that
  for every~$k \geq 2$ and  every~$D \in \class{F}$, we have
  that~$\hom_{\mathbb B}(D, \cycle[2k]) = \hom_{\mathbb B}(D,
  \cycle[2k + 1])$.
  Thus,~$\hom_{\mathbb B}(\class{F}, \cycle[2k]) = \hom_{\mathbb
    B}(\class{F}, \cycle[2k + 1])$, a contradiction.
\end{proof}

The \emph{clique number}~$\omega(G)$ of a graph~$G$ is the largest
positive integer~$k$ such that~$G$ has a~$k$-clique as a
subgraph. Clearly, two graphs~$G$ and~$H$ have the same clique
number if and only
if~$\hom_{\mathbb B}(\classofcliques,G) = \hom_{\mathbb
  B}(\classofcliques,H)$.
In contrast,  we show that the clique number cannot be
captured by any restriction of the right profile over the
Boolean semiring. The proof uses the classical result of Erd\"os that
there exist graphs of arbitrarily large chromatic number and,
simultaneously, arbitrarily large girth; see, e.g., Corollary 3.13
in~\cite{DBLP:books/daglib/0013017}.

\begin{proposition} \label{prop:clique} There is no class~$\class{F}$
  of graphs such that for every two graphs~$G$ and~$H$, we have
  that~$\omega(G) = \omega(H)$ if and only
  if~$\hom_{\mathbb B}(G,\class{F}) = \hom_{\mathbb B}(H,\class{F})$.
\end{proposition}
\begin{proof}
  Towards a contradiction, assume such a class~$\class{F}$
  exists.  Then there must be some graph~$D \in \class{F}$
  that is not an independent set, because $\omega(K_2) \not= \omega(K_3)$
  and $\hom_{\mathbb B}(K_2,I) = \hom_{\mathbb B}(K_3,I) = 0$
  for every independent set~$I \in \classofindeps$. Pick such a graph~$D$
  and let~$m = |V(D)|$. Let~$G$ be any graph of
  chromatic number larger than~$m$ and girth larger than~$3$. In
  particular~$\hom_{\mathbb B}(G,D) = 0$ and~$\omega(G) = 2$. Now
  take~$F = D \times K_2$, i.e.,~$V(F) = V(D) \times \{0,1\}$ and
  there is an edge between~$(u,a)$ and~$(v,b)$ in~$F$ (where $u,v \in V(D)$ and $a,b \in \{0,1\}$) if and only
  if~$(u,v)$ is an edge in~$F$ and~$a \not= b$. The projections
  into each component are homomorphisms into~$D$ and~$K_2$,
  respectively.  Hence~$\hom_{\mathbb B}(F,D) \not= 0$ and~$F$ is
  2-colorable. In particular~$F$ is triangle-free and, since~$D$ is
  not an independent set,~$\omega(F) = 2$. We have shown
  that~$\omega(F) = \omega(G) = 2$
  but~$\hom_{\mathbb B}(F,D) \not= \hom_{\mathbb B}(G,D)$.
\end{proof}

It is obvious that~$\omega(G) \leq \chi(G)$. The inequality
is strict for the odd cycles (since  $\omega(C_{2n+1}) = 2$ and $\chi(C_{2n+1}) = 3$) and for many other prominent graphs, such as the Petersen graph.
 There is a vast literature on graph parameters that
interpolate between the clique number and the chromatic number. A
well-known such parameter is the \emph{fractional chromatic
  number} of a graph~$G$, denoted by~$\chi_f(G)$, which arises in the study of fractional colorings  \cite{geller1976r,stahl1976n}. By definition, the quantity $\chi_f(G)$
is the optimum of a linear program that  has one real
variable~$x_U$ for each independent set~$U \in \mathscr{I}(G)$ of~$G$
and is defined as follows:
$$
\begin{array}{llll}
\text{min}  & \sum_{U \in \mathscr{I}(G)} x_U \\
\text{s.t.} & \sum_{U \in \mathscr{I}(G): v \in U} x_U \geq 1 & \text{ for all } v \in V(G), \\
& x_U \geq 0 & \text{ for all } U \in \mathscr{I}(G).
\end{array}
$$
It is known that~$\omega(G) \leq \chi_f(G) \leq \chi(G)$ and, again,
the inequalities can be strict. In particular, the inequalities are strict for odd cycles of length at least $5$, since $\chi_f(C_{2n+1}) = 2 + \frac{1}{n}$.

The fractional chromatic number of a graph can also be characterized
combinatorially. We write~$K_{a:b}$ to denote the \emph{Kneser} graph with
parameters~$a$ and~$b$, where~$a \geq 2b$. This is the graph whose
vertices are the~$b$-element subsets of~$\{1,\ldots,a\}$, and where
two~$b$-element subsets are joined by an edge if the sets are
disjoint. It is known that~$\chi_f(G)$ is the smallest rational number~$a/b$ for which there is a homomorphism from~$G$ to the
Kneser graph~$K_{a:b}$.  It follows that two graphs~$G$ and~$H$
satisfy~$\chi_f(G) = \chi_f(H)$ if and only
if~$\hom_{\mathbb B}(G,\mathscr{K}_f) = \hom_{\mathbb
  B}(H,\mathscr{K}_f)$, where~$\mathscr{K}_f$ denotes the class of all
Kneser graphs. We refer the reader to Chapter~6
of~\cite{DBLP:books/daglib/0013017} for the fascinating interplay
between graph coloring theory and the Kneser graphs.

Interestingly, the fractional chromatic number of a
graph~$G$ has a dual, called the \emph{fractional clique number}
of~$G$, denoted by~$\omega_f(G)$. By definition, this is the optimum
of the \emph{dual} of the linear program that
defines~$\chi_f(G)$. Concretely,~$\omega_f(G)$ is the optimum of
$$
\begin{array}{llll}
\text{max} & \sum_{v \in V(G)} y_v \\
\text{s.t.} & \sum_{v \in U} y_v \leq 1 & \text{ for all } U \in \mathscr{I}(G), \\
& y_v \geq 0 & \text{ for all } v \in V(G).
\end{array}
$$
By the Duality Theorem of Linear Programming, it holds
that~$\omega_f(G) = \chi_f(G)$. In view of this, and of the
characterization of~$\chi_f$ in terms of the sign of the
right profile~$\hom_{\mathbb B}(\placeholder,\mathscr{K}_f)$ restricted to Kneser graphs, it would be
interesting to know whether the equivalence relation of having the
same fractional clique number can be captured as the restriction of a
left profile. We leave this as an open problem.

\section{Concluding Remarks}
In this paper, we investigated relaxations of graph isomorphism
obtained by restricting the left profile and the
right profile to a class of graphs. We showed that these
two types of restrictions have incomparable expressive power.  In
particular, we established a number of results to the effect that
certain natural relaxations of graph isomorphism \emph{cannot} be
obtained by restricting one of these two profiles.

The work reported here motivates several different directions for
future research, including investigating extensions of this framework
to weighted graphs over semirings, as discussed in Section
\ref{sec:extensions}.  One of the ultimate goals is to characterize
the relaxations of graph isomorphism that are captured by restrictions
of the left profile or by restrictions of the
right profile.  An important result in this vein is the
Freedman-Lov\'asz-Schrijver Theorem \cite{freedman2007reflection},
which characterizes the \emph{graph parameters} (invariant numbers) of
a graph~$G$ that are equal to~$\hom(G,F)$ for some fixed weighted
graph~$F$. Much more remains to be done to obtain such
characterizations for restrictions of the left profile or
the right profile to an arbitrary class~$\class{F}$ of
graphs (not just to a single graph~$F$).

\bibliographystyle{alpha}
\bibliography{biblio}

\end{document}